\documentclass[11pt]{amsart}

\usepackage[english]{babel}
\usepackage{enumitem}
\usepackage{amsmath,amsthm,amssymb,mathrsfs,amsfonts,mathtools}
\usepackage[colorlinks]{hyperref}
\hypersetup{
  linkcolor=blue,
  citecolor=green,
  pdftitle={Large deviations and applications in homogeneous dynamics},
  pdfauthor={Taehyeong Kim and Ron Mor}
}
\newtheorem{thm}{Theorem}[section]
\newtheorem{lem}[thm]{Lemma}
\newtheorem{cor}[thm]{Corollary}
\newtheorem{prop}[thm]{Proposition}

\theoremstyle{definition}
\newtheorem{defi}[thm]{Definition}
\newtheorem*{defi*}{Definition}

\theoremstyle{remark}
\newtheorem{rem}[thm]{Remark}

\numberwithin{equation}{section}

\newcommand{\R}{\mathbb R}
\newcommand{\Z}{\mathbb Z}
\newcommand{\E}{\mathbb E}
\newcommand{\Prob}{\mathbb P}
\newcommand{\cF}{\mathcal F}
\newcommand{\Ad}{\operatorname{Ad}}
\newcommand{\inj}{\operatorname{inj}}
\newcommand{\htop}{\operatorname{ht}}
\newcommand{\supp}{\operatorname{supp}}
\newcommand{\Lip}{\operatorname{Lip}}
\newcommand{\dd}{\,\mathrm d}

\begin{document}

\title[Large deviations and applications]{Large deviations and applications in homogeneous dynamics}

\date{}

\author{Taehyeong Kim}
\address{Department of Mathematics, Brandeis University, Waltham, MA 02453, USA}
\email{taehyeongkim@brandeis.edu}

\author{Ron Mor}
\address{Department of Mathematics, Yale University, New Haven, CT 06511, USA}
\email{ron.mor@yale.edu}


\begin{abstract}
We consider the left translation action of an element $a$ on a finite-volume
semisimple homogeneous space $G/\Gamma$, where $\Ad(a)$ is diagonalizable over
$\mathbb R$.  Assuming exponential mixing, we prove a sub-Gaussian
large-deviation estimate for compactly supported smooth test functions.
As an application, we prove effective uniqueness of the measure of
maximal entropy. We also obtain quantitative pointwise equidistribution for Haar-almost every
point on each full expanding horosphere.
\end{abstract}

\maketitle

\section{Introduction}

\subsection{Background}

Let $(X,\mathcal B,\mu,T)$ be an ergodic measure-preserving system.  For every
$f\in L^1(\mu)$, Birkhoff's ergodic theorem \cite{Birkhoff1931ergodic} gives
\[
 \frac1N\sum_{j=0}^{N-1}f(T^jx)\longrightarrow
 \int_X f\dd\mu
\]
for $\mu$-almost every $x$.  Two natural refinements concern the distribution
of typical fluctuations and the probability of atypical deviations, leading
to central limit theorems and large-deviation estimates, respectively.

Large deviations, the main subject of this paper, concern the exceptional-set
measure
\[
 \mu\left\{x\in X:
 \left|\frac1N\sum_{j=0}^{N-1}f(T^jx)-\int_Xf\dd\mu\right|
 >\varepsilon\right\}.
\]

In probability theory, Cram\'er's theorem \cite{Cramer1938} is a classical
starting point for the study of large deviations of averages of independent
identically distributed random variables.  For further background and
references, see \cite{DemboZeitouni2010}.

In the dynamical setting, the analogous asymptotic question is the
exponential rate of decay of the measure of the exceptional-set above, for fixed $f$ and
$\varepsilon>0$, for large $N$.  Here we seek a finite-time exponential upper bound.
Dynamical large-deviation estimates have been established in uniformly
hyperbolic settings \cite{Kifer1990large,Young1990large} and in nonuniformly
hyperbolic settings \cite{MelbourneNicol2008}.  For a survey of these results
and further references, see \cite{Sharp2016Survey}.

In homogeneous dynamics, Khayutin \cite{Khayutin2017Large} proved
large-deviation estimates in an $S$-arithmetic setting for a semisimple
action at a non-Archimedean place. We establish finite-time sub-Gaussian
estimates for Haar measure under $\Ad$-diagonalizable translations on
real finite-volume homogeneous spaces, assuming exponential mixing.

\subsection{Main theorem}

Let $G$ be a connected semisimple real Lie group, let
$\Gamma<G$ be a lattice, and set
\[
 X=G/\Gamma.
\]
Let $m_X$ be the $G$-invariant probability measure.  Fix $a\in G$ such that
$\Ad(a)$ is diagonalizable over $\mathbb R$, acting on $X$ by left
translation.

All test functions below are real-valued.  Our only dynamical assumption is
exponential mixing.

\begin{defi}\label{exponential-mixing}
The action of $a$ on $X$ is \emph{exponentially mixing} if there are
$\sigma>0$ and $\ell\in\Z_{\geq1}$ such that, for every $n\geq0$ and
$\psi_1,\psi_2\in C_c^\infty(X)$,
\[
 \left|
 \int_X\psi_1(a^nx)\psi_2(x)\dd m_X(x)
 -m_X(\psi_1)m_X(\psi_2)
 \right|
 \ll e^{-\sigma n}
 \|\psi_1\|_{\ell,2}\|\psi_2\|_{\ell,2},
\]
where the Sobolev norm $\|\cdot\|_{\ell,2}$ is defined in
Subsection~\ref{subsec:horospherical}.
\end{defi}

If $G$ has finite center and no compact factors and $\Gamma$ is irreducible, then
$L^2_0(X)$ has a strong spectral gap
\cite[pp.~284--285]{KelmerSarnak2009}.  The resulting exponential decay of
smooth matrix coefficients gives exponential mixing for every $a$ whose
$\Ad(a)$ is diagonalizable over $\mathbb R$ and for which $\{\Ad(a)^n:n\geq0\}$ is unbounded
\cite[Corollary~2.4.4]{KM96}.
For example, this applies to
$X=\operatorname{SL}_d(\mathbb R)/\operatorname{SL}_d(\mathbb Z)$, $d\geq2$,
with $a=\operatorname{diag}(e^{t_1},\ldots,e^{t_d})$, where
$\sum_i t_i=0$ and not all $t_i$ vanish.

Our main theorem gives the following large-deviation estimate for
Birkhoff averages under exponential mixing.

\begin{thm}\label{thm:large-deviation}
Suppose that the action of $a$ on $X$ is exponentially mixing.  Then for
every $\varphi\in C_c^\infty(X)$ there is $c_\varphi>0$ such that, for every
$0<\varepsilon<1$ and every
$N\gg_\varphi\varepsilon^{-2}$,
\begin{equation}\label{eq:main-ld}
 m_X\left\{x:
 \left|\frac1N\sum_{j=0}^{N-1}\varphi(a^jx)-m_X(\varphi)\right|
 >\varepsilon\right\}
 \leq e^{-c_\varphi\varepsilon^2N}.
\end{equation}
\end{thm}

\begin{rem}
\leavevmode\par
\begin{enumerate}[label=(\roman*)]
\item Once $G$, $\Gamma$, and $a$ are fixed, the constant $c_\varphi$ may be chosen
in terms of the height of $\supp\varphi$ and a sufficiently high Sobolev norm
of $\varphi$.  It is independent of $N$ and $\varepsilon$.

\item The $\varepsilon^2$ dependence in the exponent reflects the natural Gaussian scaling:
setting $\varepsilon=t/\sqrt N$ in~\eqref{eq:main-ld} gives
\[
 m_X\left\{x:
 \left|\frac1{\sqrt N}\sum_{j=0}^{N-1}
 \bigl(\varphi(a^jx)-m_X(\varphi)\bigr)\right|>t\right\}
 \leq e^{-c_\varphi t^2}
\]
for $1\ll_\varphi t<\sqrt N$, a sub-Gaussian tail bound.
This quadratic rate also yields
the square-root entropy-defect bound in
Theorem~\ref{thm:entropy-main}.
\end{enumerate}
\end{rem}

The same estimate holds in continuous time.

\begin{cor}\label{cor:continuous-ld}
Under the hypotheses of Theorem~\ref{thm:large-deviation}, let
$(a_t)_{t\in\R}$ be a one-parameter subgroup with $a_1=a$.
For every $\varphi\in C_c^\infty(X)$ there is $c_\varphi>0$ such that,
for $0<\varepsilon<1$ and $T\gg_\varphi\varepsilon^{-2}$,
\begin{equation*}
 m_X\left\{x:
 \left|\frac1T\int_0^T\varphi(a_tx)\dd t-m_X(\varphi)\right|
 >\varepsilon\right\}
 \leq e^{-c_\varphi\varepsilon^2T}.
\end{equation*}
\end{cor}

\subsection{Effective uniqueness of the measure of maximal entropy}

Parry and Bowen proved uniqueness for irreducible shifts of finite type
and certain uniformly hyperbolic systems, respectively
\cite{Parry1964Intrinsic,Bowen1975Equilibrium}. Effective versions were
obtained by Polo for hyperbolic toral automorphisms, Kadyrov for shifts of
finite type and hyperbolic systems, and R\"uhr and Sarig for countable-state
Markov shifts
\cite{Polo2011Equidistribution,Kadyrov2015Effective,RuehrSarig2022}.

In homogeneous dynamics, leafwise entropy gives uniqueness
when the contracting and expanding horospherical subgroups
generate the ambient group
\cite[Theorem~7.9 and Corollary~7.10]{EL10}.
R\"uhr and Khayutin proved effective uniqueness for compact $p$-adic
homogeneous quotients and in an $S$-arithmetic setting, respectively
\cite{Ruehr2016Effectivity,Khayutin2017Large}.

Our Theorem~\ref{thm:large-deviation} yields the following effective uniqueness for real
finite-volume homogeneous spaces, including noncompact quotients, with
a square-root bound in the entropy gap.

\begin{thm}\label{thm:entropy-main}
Under the hypotheses of Theorem~\ref{thm:large-deviation}, for every
$\varphi\in C_c^\infty(X)$ and every $a$-invariant probability measure $\mu$
on $X$,
\begin{equation}\label{eq:entropy-main}
 |m_X(\varphi)-\mu(\varphi)|
 \ll_\varphi
 \bigl(h_{m_X}(a)-h_\mu(a)\bigr)^{1/2}.
\end{equation}
In particular, $m_X$ is the unique $a$-invariant probability measure of
maximal entropy.
\end{thm}

\subsection{Pointwise equidistribution along horospheres}\label{subsec:pointwise}

Birkhoff's theorem gives equidistribution for $m_X$-almost every initial
point. In homogeneous dynamics, one is also interested in trajectories
starting on each fixed horosphere. Let $H$ be the full expanding horospherical subgroup of
$a$, with Haar measure $m_H$, as in
Subsection~\ref{subsec:horospherical}.

For weighted diagonal flows on
$\operatorname{SL}_d(\mathbb R)/\operatorname{SL}_d(\mathbb Z)$,
Kleinbock--Shi--Weiss \cite[Theorem~1.1]{KSW17} prove pointwise
equidistribution for compactly supported smooth test functions at almost
every translate of each fixed point by their expanding unipotent
subgroup, with error
\[
 o(T^{-1/2}(\log T)^{3/2+\eta}),\qquad\eta>0.
\]
Shi \cite[Corollary~1.7]{Shi21} obtains the same rate for suitable
expanding unipotent subgroups in semisimple homogeneous spaces under
spectral-gap and expanding-cone hypotheses. For a related qualitative
result along nondegenerate analytic curves in abelian horospheres in
$\operatorname{SL}_d(\mathbb R)/\operatorname{SL}_d(\mathbb Z)$, see
Solan--Wieser \cite{SolanWieser2025}.

A further application is the following
quantitative pointwise equidistribution on every full expanding horosphere.

\begin{thm}\label{thm:pointwise-equidistribution}
Under the hypotheses of Theorem~\ref{thm:large-deviation}, let
$(a_t)_{t\in\R}$ be a one-parameter subgroup with $a_1=a$.

For every $\varphi\in C_c^\infty(X)$, the following hold.
\begin{enumerate}[label=(\roman*)]
\item For every fixed $x\in X$, for $m_H$-almost every $h\in H$,
\begin{equation*}
 \frac1T\int_0^T\varphi(a_thx)\dd t
 =m_X(\varphi)+O_\varphi\!\left(\sqrt{\frac{\log T}{T}}\right).
\end{equation*}
\item For $m_X$-almost every $x\in X$,
\begin{equation*}
 \frac1T\int_0^T\varphi(a_tx)\dd t
 =m_X(\varphi)+O_\varphi\!\left(\sqrt{\frac{\log T}{T}}\right).
\end{equation*}
\end{enumerate}

\end{thm}

In the common setting of full horospheres, our theorem improves (compared to previous results) the
logarithmic factor to $(\log T)^{1/2}$.

\subsection{Outline of the proof}

We outline the proof of Theorem~\ref{thm:large-deviation} in four steps.
Fix $\varphi\in C_c^\infty(X)$ and write
\[
 S_N(x)=\sum_{s=0}^{N-1}\varphi(a^sx),\qquad x\in X.
\]

\smallskip
\noindent\emph{Step 1: horospherical randomization.}
Let $H$ be the expanding horospherical subgroup associated with $a$ and
choose a smooth probability measure $\nu$ on $H$ supported in a small
neighborhood of the identity.
Consider the product probability space
\[
 (\Omega,\Prob)=(H^N,\nu^{\otimes N}).
\]
Write $\mathbf h=(h_0,\ldots,h_{N-1})$ for the independent coordinates and set
\[
 P_0=e,\qquad
 P_k=(a^{-(k-1)}h_{k-1}a^{k-1})\cdots(a^{-1}h_1a)h_0,
 \quad 1\leq k\leq N.
\]
These perturbations enter at contracting scales.  For base points $x$ of
sufficiently small height, we seek a large-deviation estimate on this
probability space:
\begin{equation}\label{eq:outline-randomized-deviation}
 \Prob\!\left\{|S_N(P_Nx)-Nm_X(\varphi)|>\varepsilon N\right\}
 \ll_\varphi e^{-c_\varphi\varepsilon^2N}.
\end{equation}
The large-height region is handled by the cusp-tail estimate.  Haar
invariance then gives the desired bound for $S_N(x)$.

\smallskip
\noindent\emph{Step 2: reduction to martingale estimates.}
For fixed $x$, reveal the coordinates one at a time:
\[
 \cF_k=\sigma(h_0,\ldots,h_{k-1}),\qquad
 X_k=\E\bigl(S_N(P_Nx)\mid\cF_k\bigr),
 \quad 0\leq k\leq N,
\]
where $\cF_0$ is trivial.  Then $(X_k)$ is a martingale and
$X_N=S_N(P_Nx)$.  On the region in Step 1, it suffices to show
\begin{align}
  X_0&\approx Nm_X(\varphi),
  \label{eq:outline-initial-expectation}\\
  |X_{k+1}-X_k|&\ll_\varphi1,
  \qquad 0\leq k<N.
  \label{eq:outline-bounded-differences}
\end{align}
Azuma's inequality (Proposition~\ref{prop:azuma}) then gives
\eqref{eq:outline-randomized-deviation}.

\smallskip
\noindent\emph{Step 3: convolution replaces the product-space integral.}
Let $\nu_j$ be the law of $a^{-j}h_ja^j$.  For $k<N$, the remaining random
product has law
\[
 \mu_k^{(N)}=\nu_{N-1}*\cdots*\nu_k.
\]
For $k<N$, the integral over the $N-k$ hidden coordinates becomes an
integral over a single variable $h\in H$:
\[
 \begin{aligned}
 X_k
 &=\int_{H^{N-k}}S_N(P_N(\mathbf h)x)
       \dd\nu(h_{N-1})\cdots\dd\nu(h_k)\\
 &=\int_H S_N(hP_kx)\dd\mu_k^{(N)}(h)
 =\sum_{s=0}^{N-1}\int_H\varphi(a^shP_kx)\dd\mu_k^{(N)}(h).
 \end{aligned}
\]
After rescaling, a fixed smooth factor $\nu$ gives uniformly controlled
densities (Lemma~\ref{lem:uniform-densities}).  The key formula
\eqref{eq:future-conditional-averages} reduces the estimates for the
summands above to integrals of the form
\[
 \int_H f(h)\varphi(a^thy)\dd m_H(h),
\]
with $t\in\Z_{\geq0}$, $y\in X$, and a smooth probability density $f$.
This is the bridge to dynamics.

\smallskip
\noindent\emph{Step 4: dynamical inputs.}
For base points $y$ in a fixed compact set $K$, effective equidistribution
gives, for smooth probability densities $f$ supported in $B_1^H$,
\[
 \left|\int_H f(h)\varphi(a^thy)\dd m_H(h)-m_X(\varphi)\right|
 \ll_{\varphi,K}\|f\|_{C^\ell}e^{-\delta_1t}.
\]
For the base points in the region of Step 1, summing the effective
equidistribution estimates gives the schematic relation
\[
 X_0=\sum_{s=0}^{N-1}\int_H\varphi(a^shx)\dd\mu_0^{(N)}(h)
 \approx Nm_X(\varphi),
\]
which yields
\eqref{eq:outline-initial-expectation}.
To prove \eqref{eq:outline-bounded-differences}, fix $0\leq k<N$ and
estimate the change $X_{k+1}-X_k$ caused by revealing $h_k$.
For \emph{past times} $s\leq k$, contraction makes the effect of revealing
$h_k$ decay exponentially in $k-s$.
For \emph{future times} $s>k$, effective equidistribution and quantitative
non-divergence give
compact-return averaging (Proposition~\ref{prop:compact-return-average}),
with exponential decay from the first possible visit to $\supp\varphi$.
For some decay rates $\chi,\delta_2>0$, summing these bounds gives
\[
 |X_{k+1}-X_k|
 \ll_\varphi
 \underbrace{\sum_{r\geq0}e^{-\chi r}}_{\text{past times}}
 +\underbrace{\sum_{r\geq0}e^{-\delta_2r}}_{\text{future times}}
 \ll_\varphi1.
\]

\medskip
\noindent\textbf{Structure of the paper.}\ Section~\ref{sec:inputs} collects the
quantitative inputs.  It fixes the expanding horospherical subgroup, height,
and norms and establishes effective equidistribution, quantitative
non-divergence, and the cusp-tail estimate.  Section~\ref{sec:convolution} constructs the
contracting convolutions and proves the uniform density and partition
estimates.  Section~\ref{sec:compact-return} combines equidistribution and
non-divergence into the compact-return averaging proposition.
Section~\ref{sec:martingale} constructs the martingale and proves the
uniform bounded-difference estimate.  Section~\ref{sec:ld-proof} applies
Azuma's inequality and the cusp-tail estimate to prove
Theorem~\ref{thm:large-deviation} and Corollary~\ref{cor:continuous-ld},
and then proves the pointwise
equidistribution in Theorem~\ref{thm:pointwise-equidistribution}. Finally,
Section~\ref{sec:entropy} establishes the local-entropy and Bowen-volume
estimates and proves Theorem~\ref{thm:entropy-main}.

\medskip
\noindent\textbf{Conventions.}
For nonnegative quantities $A$ and $B$, we write $A\ll B$ if $A\leq CB$ for
some constant $C>0$, write $A\gg B$ if $B\ll A$, and write $A\asymp B$ if
both $A\ll B$ and $B\ll A$.  The constant $C$ is called the implied
constant.  Unless otherwise indicated, the implied constants may depend on
$G$, $\Gamma$, $a$, and the auxiliary choices fixed throughout the paper;
additional dependence is recorded by subscripts.

\medskip
\noindent\textbf{AI disclosure.}
OpenAI Codex was used only as a supplementary tool for checking parts of the mathematical reasoning and calculations and for improving the presentation of the manuscript. The research questions, mathematical ideas, proofs, and results were developed by the authors. The authors reviewed and verified all AI-assisted material and retain full responsibility for the final manuscript.

\medskip
\noindent\textbf{Acknowledgments.}\ The authors are grateful to Elon Lindenstrauss
for suggesting this project, for his encouragement, and for sharing many
ideas and insights that shaped its development.
They also thank Dmitry Kleinbock for pointing out the application to
pointwise equidistribution in Subsection~\ref{subsec:pointwise}.

\section{Quantitative estimates along horospheres}\label{sec:inputs}

This section develops the quantitative horospherical inputs used in the
proof.  We first set up the expanding horospherical subgroup and its
contraction, and then record effective equidistribution, quantitative
non-divergence, and a cusp-tail estimate.

\subsection{The horospherical subgroup, height, and norms}
\label{subsec:horospherical}

We fix the expanding horospherical subgroup associated with $a$, record the
metric and derivative bounds for conjugation, and introduce the height and
differentiable norms used below.

Write $\mathfrak g=\operatorname{Lie}(G)$.  Following
\cite[\S~1.3]{KM96}, set
\[
 \mathfrak h
 =\{Y\in\mathfrak g:\Ad(a^{-n})Y\longrightarrow0
     \text{ as }n\to\infty\},
\]
and let $H$ be the connected Lie subgroup with Lie algebra $\mathfrak h$.
Equivalently,
\[
 H=\{h\in G:a^{-n}ha^n\longrightarrow e\text{ as }n\to\infty\}.
\]
For every eigenvalue $\beta$ of $\Ad(a)$, let
\[
 E_\beta=\ker\bigl(\Ad(a)-\beta\operatorname{id}_{\mathfrak g}\bigr).
\]
Since $\Ad(a)$ is diagonalizable over $\R$,
\[
 \mathfrak g=\bigoplus_\beta E_\beta,
 \qquad
 \mathfrak h=\bigoplus_{|\beta|>1}E_\beta.
\]
In particular, $H$ is connected, simply connected, nilpotent, and hence
unimodular.  Choose an inner product on $\mathfrak g$ for which the
eigenspaces $E_\beta$ are mutually orthogonal.  Let $d_G$ be the
corresponding right-invariant Riemannian distance on $G$, let $d_X$ be the
induced quotient metric, and let $d_H$ be the intrinsic Riemannian distance
on $H$ obtained by restricting the inner product to $\mathfrak h$.
Write $\|\cdot\|$ for these inner-product norms and the corresponding
operator norms.
Denote by $B_r^G$ and $B_r^H$ the balls of radius $r$ around the identity in $G$ and $H$, respectively.
Let $m_H$ be the normalized Haar measure on $H$ so that $m_H(B_1^H)=1$.

For $n\in\mathbb Z$ and $h\in H$, put
\[
 \lambda_n(h)=a^nha^{-n},\qquad
 J_n=\left|\det\left(\Ad(a)^n|_{\mathfrak h}\right)\right|.
\]
Thus $m_H(\lambda_n(E))=J_nm_H(E)$ for measurable $E\subset H$,
$J_n=J_1^n$, and $J_n\geq1$ for $n\geq0$.
For $x\in X$, let $\pi_x:G\to X$ be the map given by $g\mapsto gx$ and define
\[
 \inj(x)=\sup\{r>0:\pi_x\text{ is injective on }B_r^G\}.
\]
We use the height
\begin{equation*}
 \htop(x)=\max\{1,\inj(x)^{-1}\}.
\end{equation*}
In particular, for every compact $C\subset G$, all $g\in C$, and $x\in X$,
\begin{equation}\label{eq:height-compact-change}
 \htop(gx)\asymp_C\htop(x).
\end{equation}

Fix bases $Y_1,\ldots,Y_d$ of $\mathfrak h$ and
$Z_1,\ldots,Z_D$ of $\mathfrak g$.
For $f\in C^\infty(H)$, $\varphi\in C^\infty(X)$, $h\in H$, and $x\in X$, define
\[
 Y_i f(h)=\left.\frac{\dd}{\dd t}\right|_{t=0} f(\exp(tY_i)h), \qquad Z_i\varphi(x)=\left.\frac{\dd}{\dd t}\right|_{t=0} \varphi(\exp(tZ_i)x).
\]
For $I=(i_1,\ldots,i_j)\in\{1,\ldots,d\}^j$ and $J=(j_1,\ldots,j_r)\in\{1,\ldots,D\}^r$, set
\[
 Y_I=Y_{i_1}\cdots Y_{i_j}, \qquad Z_J=Z_{j_1}\cdots Z_{j_r}.
\]
For $j=r=0$, both operators are the identity. Write $\|\cdot\|_2$ for the $L^2(X)$-norm, and we use the following notation: for $f\in C^\infty(H)$ and $\varphi \in C^\infty(X)$
\begin{align*}
 \|f\|_{C^\ell} &=\max_{\substack{0\leq j\leq\ell\\ I\in\{1,\ldots,d\}^j}} \|Y_If\|_\infty, \quad  \|\varphi\|_{C^\ell} =\max_{\substack{0\leq r\leq\ell\\ J\in\{1,\ldots,D\}^r}} \|Z_J\varphi\|_\infty, \\
 \|\varphi\|_{\ell,2} &=\sum_{\substack{0\leq r\leq \ell \\ J\in\{1,\ldots,D\}^r}} \|Z_J\varphi\|_2, \quad  \|\varphi\|_{\Lip} =\sup_{x\neq y}\frac{|\varphi(x)-\varphi(y)|}{d_X(x,y)}.
\end{align*}

The following lemma collects the metric and derivative estimates used below.
\begin{lem}\label{lem:conjugation-adjoint}
For $g\in G$ normalizing $H$, write $C_g(h)=ghg^{-1}$.
The following properties hold.
\begin{enumerate}[label=(\roman*)]
\item For $h,h'\in H$,
\[
 d_H(C_g(h),C_g(h'))
 \leq\|\Ad(g)|_{\mathfrak h}\|\,d_H(h,h').
\]

\item There exists $\chi>0$ such that, for every
$n\in\Z_{\geq0}$ and $Y\in\mathfrak h$,
\[
 \|\Ad(a^{-n})Y\|\leq e^{-\chi n}\|Y\|.
\]
Hence, for $h,h'\in H$,
\[
 d_H(\lambda_{-n}(h),\lambda_{-n}(h'))
 \leq e^{-\chi n}d_H(h,h'),
\]
and, for every $r>0$,
\begin{equation}\label{eq:contraction}
 \lambda_{-n}(B_r^H)\subset B_{re^{-\chi n}}^H.
\end{equation}

\item For $f\in C^\infty(H)$, $I=(i_1,\ldots,i_j)$, and $u,v,h\in H$,
\begin{align*}
 Y_I(f\circ C_g)(h)
 &=\bigl[(\Ad(g)Y_{i_1})\cdots(\Ad(g)Y_{i_j})f\bigr](C_g(h)),\\
 Y_I[f(u^{-1}\,\cdot\,)](h)
 &=\bigl[(\Ad(u^{-1})Y_{i_1})\cdots
         (\Ad(u^{-1})Y_{i_j})f\bigr](u^{-1}h),
 \\
 Y_I[f(\,\cdot\,v^{-1})](h)
 &=(Y_If)(hv^{-1}).
\end{align*}

\end{enumerate}
\end{lem}

\begin{proof}
For (i), let $\gamma:[0,1]\to H$ be a piecewise $C^1$ curve from $h$
to $h'$.  At each differentiability point, choose $Y_t\in\mathfrak h$ such
that
\[
 \dot\gamma(t)=\left.\frac{\dd}{\dd s}\right|_{s=0}
 \exp(sY_t)\gamma(t).
\]
The identity
$C_g(\exp(sY_t)\gamma(t))
=\exp(s\Ad(g)Y_t)C_g(\gamma(t))$
shows that the right-invariant metric gives speeds $\|Y_t\|$ and
$\|\Ad(g)Y_t\|$ for $\gamma$ and $C_g\circ\gamma$, respectively.  Hence
\[
 \operatorname{length}(C_g\circ\gamma)
 =\int_0^1\|\Ad(g)Y_t\|\dd t
 \leq\|\Ad(g)|_{\mathfrak h}\|\operatorname{length}(\gamma).
\]
Taking the infimum over $\gamma$ proves (i).

For (ii), choose $\chi>0$ such that $|\beta|\geq e^\chi$ for every
eigenvalue $\beta$ of $\Ad(a)$ with $|\beta|>1$, and write
$Y=\sum_{|\beta|>1}Y_\beta$.  Orthogonality of the eigenspaces gives
\[
 \|\Ad(a^{-n})Y\|^2
 =\sum_{|\beta|>1}|\beta|^{-2n}\|Y_\beta\|^2
 \leq e^{-2\chi n}\sum_{|\beta|>1}\|Y_\beta\|^2
 =e^{-2\chi n}\|Y\|^2.
\]
Apply (i) with $g=a^{-n}$ to obtain the distance contraction.
The ball inclusion follows due to $\lambda_{-n}(e)=e$.

For (iii), the exponential conjugation identity gives
\begin{align*}
 Y_i(f\circ C_g)(h)
 &=\left.\frac{\dd}{\dd t}\right|_{t=0}
   f\bigl(\exp(t\Ad(g)Y_i)C_g(h)\bigr)\\
 &=\bigl((\Ad(g)Y_i)f\bigr)(C_g(h)).
\end{align*}
For left translation, use
$u^{-1}\exp(tY_i)h=\exp(t\Ad(u^{-1})Y_i)u^{-1}h$;
for right translation, use
$\exp(tY_i)hv^{-1}=\exp(tY_i)(hv^{-1})$.
Differentiating these identities and iterating gives the ordered formulas
for $Y_I$.

\end{proof}

\subsection{Effective equidistribution, quantitative non-divergence, and cusp tails}

We collect the three quantitative inputs used in the proof: effective
equidistribution of expanding horospherical translates, quantitative
non-divergence for translates based at a compact set, and a polynomial tail
bound for the height under Haar measure.

\begin{prop}\label{prop:effective-equidistribution}
Suppose that the action of $a$ on $X$ is exponentially mixing, and let $\ell\in\Z_{\geq 1}$ be as in Definition \ref{exponential-mixing}.
There are $\delta_1>0$ and $\omega>0$ such that for every $n\geq0$, $x\in X$, $\varphi\in C_c^\infty(X)$, and $f\in C_c^\infty(H)$ supported in $B_1^H$,
\begin{align*}
 &\left|\int_H f(h)\varphi(a^n h x)\dd m_H(h)
 -m_H(f)m_X(\varphi)\right|\\
 &\qquad\ll
 \max\!\left\{\|\varphi\|_{C^1},\|\varphi\|_{\ell,2}\right\}
 \|f\|_{C^\ell}\htop(x)^\omega e^{-\delta_1 n}.
\end{align*}
\end{prop}

\begin{proof}
The proof is based on the thickening argument, which goes back to
\cite{KM96}; the proposition follows by applying the proof of
\cite[Theorem~2.5]{KM20} at integer times.
\end{proof}

\begin{prop}
\label{prop:quantitative-nondivergence}
There is $\kappa_1>0$ such that, for every compact $K\subset X$, $R>0$,
$n\geq0$, $x\in K$, and $r>0$,
\[
 m_H\left\{h\in B_R^H:\htop(\lambda_n(h)x)>r^{-1}\right\} \ll_{K,R}r^{\kappa_1}.
\]
\end{prop}

\begin{proof}
By the reductions in \cite[proofs of Lemmas~3.6--3.7]{Shi21},
we may assume that $G$ is center-free with no compact factors and that
$\Gamma$ is irreducible; compact quotients are immediate.
It remains to consider the nonuniform higher-rank arithmetic case
and the real rank-one case.

In the arithmetic case, use the embedding into
$\operatorname{SL}_M(\mathbb R)/\operatorname{SL}_M(\mathbb Z)$ provided by the reduction and apply \cite[Theorem~5.2]{LMMS2024quantitative}.
Compactness of the base-point set and the identity $\lambda_n(e)=e$
give the uniform lower bounds for the covolumes of all primitive
subgroups required in its statement.
The resulting measure estimate controls the set of elements in the horospherical ball for which the corresponding lattice leaves the Mahler compact set $K_\epsilon$ of lattices with no nonzero vector shorter than $\epsilon$.
By \cite[Proposition~3.5]{KM2012effective}, the injectivity radius on $K_\epsilon$ is uniformly bounded below by $\gg\epsilon^k$.
Together with the pointwise injectivity-radius comparison in the proof
of \cite[Lemma~3.7]{Shi21}, the measure estimate therefore gives
the required power-law bound for $\htop$ on the original space.

In real rank one, the estimate follows from
\cite[Lemma~3.3]{Shi21}, particularly equations~(3.8)--(3.10).
Combining the two cases via Shi's reduction proves the
proposition.
\end{proof}

\begin{prop}\label{prop:cusp-tail}
There is $\kappa_2>0$ such that, for all $r>0$,
\begin{equation*}
 m_X\left\{x\in X:\htop(x)>r^{-1}\right\}\ll r^{\kappa_2}.
\end{equation*}
\end{prop}

\begin{proof}
If $X$ is compact, the assertion is immediate.  In the noncompact
case, fix $x_0\in X$. By the elementary conjugation argument in the
proof of \cite[Lemma~3.14]{BH26}, for every $x\in X$,
\[
\left|\log \inj(x) \right| \ll d_X(x,x_0)+1.
\]
Hence, for suitable $C_1,C_2>0$,
\[
 \left\{x\in X:\htop(x)>r^{-1}\right\} \subset \left\{x\in X:d_X(x,x_0)>C_1\log(r^{-1})-C_2\right\}.
\]
It follows from \cite[Proposition~5.1]{KM1999logarithm} (see also \cite[Section~10.3]{KM1999logarithm}) that there is $c>0$ such that, for every $s\geq0$,
\[
 m_X\left\{x\in X:d_X(x,x_0)>s\right\}\ll e^{-cs}.
\]
Let $s=C_1\log(r^{-1})-C_2>0$ and $\kappa_2=cC_1$.
Then the proposition follows immediately for $r$ small enough so that $s>0$.
The remaining range of $r$ is absorbed by increasing the implied constant.
\end{proof}

\section{Contracting convolutions}\label{sec:convolution}

We construct contracting convolutions on $H$ that provide the perturbations
used in the martingale argument.  We prove that their rescaled tails have
uniformly controlled smooth densities and construct a uniformly controlled
partition of the corresponding expanding images of horospherical balls.

Fix a contraction rate $\chi>0$ as in
Lemma~\ref{lem:conjugation-adjoint}(ii), and choose $r>0$ so that
\[
 r\sum_{j\geq0}e^{-\chi j}<1.
\]
Fix a nonnegative $\eta\in C_c^\infty(H)$ with
$\int_H\eta\dd m_H=1$ and $\supp\eta\subset B_r^H$.

Let $\nu$ be the probability measure defined by
\[
 \dd\nu(h)=\eta(h)\dd m_H(h),
\]
and, for $j\geq0$, define
\[
 \nu_j=(\lambda_{-j})_*\nu.
\]
Thus, for every Borel set $E\subset H$ and $f\in C_c(H)$,
\[
 \nu_j(E)=\nu(\lambda_j(E)),
 \qquad
 \int_Hf(h)\dd\nu_j(h)
 =\int_Hf(\lambda_{-j}(h))\eta(h)\dd m_H(h).
\]
Since $m_H(\lambda_j(E))=J_jm_H(E)$, changing variables by $\lambda_j$
gives
\[
 \frac{\dd\nu_j}{\dd m_H}(h)=J_j\eta(\lambda_j(h)),
 \qquad
 \supp\nu_j=\lambda_{-j}(\supp\eta).
\]
For probability measures $\mu_1,\mu_2$ on $H$ and $F\in C_c(H)$, we use the convolution convention
\[
 \int F\dd(\mu_1*\mu_2) = \iint F(uv)\dd\mu_1(u)\dd\mu_2(v).
\]
For $0\leq k<N$ set
\begin{equation}\label{eq:mu-k}
 \mu_k^{(N)}=\nu_{N-1}*\nu_{N-2}*\cdots*\nu_k.
\end{equation}
We also set $\mu_N^{(N)}=\delta_e$.  Since $\lambda_k$ is an
automorphism of $H$, for $j\geq k$ we have
\[
 (\lambda_k)_*\nu_j
 =(\lambda_k\circ\lambda_{-j})_*\nu
 =(\lambda_{-(j-k)})_*\nu
 =\nu_{j-k}.
\]
It follows that
\begin{equation}\label{eq:rescaled-mu}
 (\lambda_k)_*\mu_k^{(N)}
 =\mu_0^{(N-k)}
 =\mu_1^{(N-k)}*\nu.
\end{equation}
Thus rescaling at time $k$ isolates the fixed smooth factor $\nu$.

The next lemma records the uniform density estimate used below.
\begin{lem}\label{lem:uniform-densities}
Let $\ell,N\in\Z_{\geq 1}$,
$0\leq k\leq N-1$, and $s\in\Z_{\geq k}$.  Put $q=s-k$.
The measure $\zeta=(\lambda_s)_*\mu_k^{(N)}$ is of the form
\begin{equation*}
 \frac{\dd\zeta}{\dd m_H}(h)=J_q^{-1}F(h),
\end{equation*}
for some $F\in C_c^\infty(H)$ satisfying
\[
 F\geq0,\quad \supp F\subset \lambda_q(B_1^H),\quad
 \int_HF\dd m_H=J_q,\quad
 \|F\|_{C^\ell}\ll_\ell1.
\]
If $k\leq N-2$, the same conclusion holds for every $v\in\supp\nu$ when the last factor $\nu_k$ in \eqref{eq:mu-k} is replaced by $\delta_{\lambda_{-k}(v)}$.
\end{lem}

\begin{proof}
We first prove the assertion when $s=k$.
By \eqref{eq:rescaled-mu}, the rescaled measure has a smooth density $F_0$ given by
\begin{equation}\label{eq:rescaled-convolution}
 \frac{\dd(\lambda_k)_*\mu_k^{(N)}}{\dd m_H}(h)
 =F_0(h)
 =\int_H\eta(u^{-1}h)\dd\mu_1^{(N-k)}(u).
\end{equation}
By the choice of $\eta$, the support identity above, and
\eqref{eq:contraction}, for every $j\geq0$,
\[
 \supp\nu_j
 \subset\lambda_{-j}(B_r^H)
 \subset B_{re^{-\chi j}}^H.
\]
The right invariance of $d_H$ therefore gives
\[
 \supp\mu_1^{(N-k)}
 \subset B_{r\sum_{j\geq1}e^{-\chi j}}^H,
 \qquad
 \supp F_0
 \subset B_{r\sum_{j\geq0}e^{-\chi j}}^H
 \subset B_1^H.
\]
By Lemma~\ref{lem:conjugation-adjoint}(iii), differentiating
\eqref{eq:rescaled-convolution} gives, for $I=(i_1,\ldots,i_j)$,
\[
 Y_IF_0(h)
 =\int_H
 \bigl[(\Ad(u^{-1})Y_{i_1})\cdots
        (\Ad(u^{-1})Y_{i_j})\eta\bigr](u^{-1}h)
 \dd\mu_1^{(N-k)}(u).
\]
The integration variable lies in $B_1^H$.
Since $\Ad$ is continuous and $\overline{B_1^H}$ is compact, the constant
\[
 M=\sup_{u\in\overline{B_1^H}}
 \|\Ad(u^{-1})|_{\mathfrak h}\|
\]
is finite and independent of $N$ and $k$.
We can write $\Ad(u^{-1})Y_{i_j}$ as a linear combination of the basis elements, with coefficients bounded by $\ll M$, and then applying linearity it follows that
\[
 \|Y_IF_0\|_\infty
 \ll_j M^j\|\eta\|_{C^j}.
\]
Taking the maximum over $0\leq j\leq\ell$ gives
\begin{equation}\label{eq:uniform-rescaled-density}
 \|F_0\|_{C^\ell}
 \ll_\ell\|\eta\|_{C^\ell}
 \ll_\ell1.
\end{equation}
Since $(\lambda_k)_*\mu_k^{(N)}$ is a probability measure,
\[
 F_0\geq0,
 \qquad
 \int_HF_0\dd m_H=1.
\]
This proves the assertion for $s=k$, with $F=F_0$.

Now let $s>k$.  Since $q=s-k$, $\lambda_s=\lambda_q\circ\lambda_k$, and $m_H(\lambda_q(E))=J_qm_H(E)$,
\begin{equation*}
 \frac{\dd(\lambda_s)_*\mu_k^{(N)}}{\dd m_H}(h)
 =J_q^{-1}F_0(\lambda_{-q}(h)).
\end{equation*}
Thus $F=F_0\circ\lambda_{-q}$ satisfies
\[
 F\geq0,\qquad
 \supp F\subset\lambda_q(B_1^H),\qquad
 \int_HF\dd m_H=J_q.
\]
Since $\lambda_{-q}=C_{a^{-q}}$,
Lemma~\ref{lem:conjugation-adjoint}(iii) gives the ordered derivative identity
\[
 Y_IF(h)
 =\bigl[(\Ad(a^{-q})Y_{i_1})\cdots
        (\Ad(a^{-q})Y_{i_j})F_0\bigr](\lambda_{-q}(h)).
\]
Lemma~\ref{lem:conjugation-adjoint}(ii) and
\eqref{eq:uniform-rescaled-density} therefore give
\begin{align*}
 \|Y_IF\|_\infty&\ll_j e^{-j\chi q}\|F_0\|_{C^j},\\
 \|F\|_{C^\ell}&\ll_\ell\|F_0\|_{C^\ell}\ll_\ell1.
\end{align*}

Suppose $k\leq N-2$. Since $\nu_1$ has smooth density
$\eta_1(h)=J_1\eta(\lambda_1(h))$, the convolution
$\mu_1^{(N-k)}=\mu_2^{(N-k)}*\nu_1$ has a smooth density $F_1$
with respect to $m_H$, given by
\[
 \frac{\dd\mu_1^{(N-k)}}{\dd m_H}(h)=F_1(h)
 =\int_H\eta_1(u^{-1}h)\dd\mu_2^{(N-k)}(u).
\]
Using the left-translation identity in Lemma~\ref{lem:conjugation-adjoint}(iii)
for $\eta_1$ and repeating the preceding derivative estimate gives
\[
 \supp F_1\subset B_{r\sum_{j\geq1}e^{-\chi j}}^H,
 \qquad
 \|F_1\|_{C^\ell}\ll_\ell1.
\]

For $v\in\supp\nu$, unimodularity of $H$ gives
\begin{align*}
 & (\lambda_k)_*
 \left(\nu_{N-1}*\cdots*\nu_{k+1}*
 \delta_{\lambda_{-k}(v)}\right)
 =\mu_1^{(N-k)}*\delta_v,\\
 & \frac{\dd(\mu_1^{(N-k)}*\delta_v)}{\dd m_H}(h)
 =F_1(hv^{-1}),\\
 & \supp\bigl(F_1(\,\cdot\,v^{-1})\bigr)
 \subset B_{r\sum_{j\geq0}e^{-\chi j}}^H
 \subset B_1^H.
\end{align*}
The right-translation identity in Lemma~\ref{lem:conjugation-adjoint}(iii)
and the bijectivity of $h\mapsto hv^{-1}$ give
\[
 \|F_1(\,\cdot\,v^{-1})\|_{C^\ell}
 =\|F_1\|_{C^\ell}\ll_\ell1.
\]
Applying the preceding $\lambda_q$-pushforward calculation to this density, with
\[
 F_v(h)=F_1\bigl(\lambda_{-q}(h)v^{-1}\bigr),
\]
proves the point-mass assertion.
\end{proof}

We also need a uniformly controlled partition of the expanding set $\lambda_q(B_1^H)$.

\begin{lem}\label{lem:partition}
For every $\ell\in\Z_{\geq1}$ and $q\in\Z_{\geq0}$, there exist a finite index set $\mathcal I_q$, points $g_i\in\lambda_q(B_1^H)$, and functions $\psi_i\in C_c^\infty(H)$ for $i\in\mathcal I_q$, such that
\begin{enumerate}[label=(\roman*)]
\item
\[
 0\leq\psi_i\leq1,
 \qquad
 \sum_{i\in\mathcal I_q}\psi_i=1
 \quad\text{on }\lambda_q(B_1^H);
\]
\item
\[
 \supp\psi_i\subset B_1^Hg_i,
 \qquad
 \|\psi_i\|_{C^\ell}\ll_\ell1;
\]
\item the sets $B_{1/10}^Hg_i$, $i\in\mathcal I_q$, are pairwise disjoint;
\item $\#\mathcal I_q\asymp J_q$.
\end{enumerate}
\end{lem}

\begin{proof}
Choose a maximal $1/4$-separated set
\[
 \{g_i:i\in\mathcal I_q\}\subset\lambda_q(B_1^H).
\]
Then
\[
 \lambda_q(B_1^H) \subset\bigcup_{i\in\mathcal I_q}B_{1/4}^Hg_i, \qquad B_{1/10}^Hg_i\cap B_{1/10}^Hg_j=\varnothing
 \quad(i\neq j).
\]
By \eqref{eq:contraction},
\[
 B_1^H\lambda_q(B_1^H) =\lambda_q\left(\lambda_{-q}(B_1^H)B_1^H\right) \subset\lambda_q(B_2^H).
\]
Since $m_H(B_1^H)=1$,
\[
 J_q =m_H\left(\lambda_q(B_1^H)\right) \leq \#\mathcal I_q\,m_H(B_{1/4}^H),
\]
whereas
\[
 \#\mathcal I_q\,m_H(B_{1/10}^H) = m_H\left(\bigsqcup_{i\in\mathcal I_q}B_{1/10}^Hg_i\right) \leq J_qm_H(B_2^H).
\]
Thus
\[
 \#\mathcal I_q\asymp J_q.
\]

Choose $\psi\in C_c^\infty(H)$ such that
\[
 0\leq\psi\leq1,\qquad \psi=1\text{ on }B_{1/3}^H,\qquad \supp\psi\subset B_{1/2}^H,
\]
and put
\[
 \Psi_q(h)=\sum_{j\in\mathcal I_q}\psi(hg_j^{-1}).
\]
Then
\[
 \Psi_q\geq1\quad\text{on }\lambda_q(B_1^H).
\]
The pairwise disjointness above gives
\[
 \sup_{h\in H} \#\{j\in\mathcal I_q:h\in B_{1/2}^Hg_j\}\ll1.
\]
By the right-translation identity in Lemma~\ref{lem:conjugation-adjoint}(iii)
and the bounded overlap,
\[
 \|\Psi_q\|_{C^\ell}\ll \|\psi\|_{C^\ell} \ll_\ell 1.
\]

Choose $\tau\in C^\infty([0,\infty))$ such that
\[
 0\leq s\tau(s)\leq1, \qquad s\tau(s)=1\quad(s\geq1),
\]
and define
\[
 \psi_i(h)=\psi(hg_i^{-1})\tau(\Psi_q(h)).
\]
Then
\[
 0\leq\psi_i\leq1, \qquad \sum_{i\in\mathcal I_q}\psi_i =\Psi_q\tau(\Psi_q)=1 \quad\text{on }\lambda_q(B_1^H),
\]
and
\[
 \supp\psi_i \subset B_{1/2}^Hg_i \subset B_1^Hg_i, \qquad \|\psi_i\|_{C^\ell}\ll_\ell1.
\]
\end{proof}

\section{Averaging after a compact return}\label{sec:compact-return}

In this section we combine effective equidistribution with quantitative
non-divergence to prove a uniform estimate for horospherical averages after
a compact return.  The partition from Lemma~\ref{lem:partition} separates
the controlled-height cells, where equidistribution applies, from the
remaining cells, whose total contribution is controlled by non-divergence.

\begin{prop}\label{prop:compact-return-average}
Let $\ell$ be as in Proposition~\ref{prop:effective-equidistribution} and let $\varphi\in C_c^\infty(X)$.  There is $\delta_2>0$ with the following property.  Suppose $q,t\in\mathbb Z_{\geq0}$, $y\in X$, and $F\in C_c^\infty(H)$ satisfies
\[
 F\geq0,\qquad \supp F\subset\lambda_q(B_1^H),\qquad \int_HF\dd m_H=J_q.
\]
If there is $h_*\in\lambda_q(B_1^H)$ such that $h_*y\in\supp\varphi$, then
\begin{equation}\label{eq:compact-return-average}
 \left| \frac1{J_q}\int_HF(h)\varphi(a^thy)\dd m_H(h)-m_X(\varphi) \right| \ll_\varphi\|F\|_{C^\ell}e^{-\delta_2t}.
\end{equation}
\end{prop}

\begin{proof}
Let
\[
 \{(g_i,\psi_i):i\in\mathcal I_q\}
\]
be given by Lemma~\ref{lem:partition}.  Since
$\sum_i\psi_i=1$ on $\supp F$ and $\int_HF\dd m_H=J_q$,
\begin{align*}
 &\frac{1}{J_q}\int_HF(h)\varphi(a^thy)\dd m_H(h)-m_X(\varphi)\\
 &\quad=\frac{1}{J_q}\sum_{i\in\mathcal I_q}
 \int_HF(h)\psi_i(h)
 \bigl(\varphi(a^thy)-m_X(\varphi)\bigr)\dd m_H(h).
\end{align*}
Fix $\theta>0$, to be chosen at the end, and split $\mathcal I_q$ into
\[
 \mathcal G =\{i\in\mathcal I_q:\htop(g_i y)\leq e^{\theta t}\}, \qquad \mathcal B=\mathcal I_q\setminus\mathcal G.
\]
We estimate the $\mathcal G$-sum by effective equidistribution and the $\mathcal B$-sum by quantitative non-divergence.

First let $i\in\mathcal G$ and put
\[
 f_i(u)=F(ug_i)\psi_i(ug_i).
\]
By Lemma~\ref{lem:partition}, the right-translation identity in
Lemma~\ref{lem:conjugation-adjoint}(iii), and the product rule,
\[
 \supp f_i\subset B_1^H, \qquad \|f_i\|_{C^\ell}\ll_\ell\|F\|_{C^\ell}.
\]
The change of variables $h=ug_i$ and Proposition~\ref{prop:effective-equidistribution} give
\begin{align*}
 &\left| \int_HF(h)\psi_i(h) \left(\varphi(a^thy)-m_X(\varphi)\right)\dd m_H(h) \right|\\
 &=\left|
 \int_Hf_i(u)
 \bigl(\varphi(a^tug_i y)-m_X(\varphi)\bigr)\dd m_H(u)
 \right|\\
 &\ll
 \max\!\left\{\|\varphi\|_{C^1},\|\varphi\|_{\ell,2}\right\}
 \|f_i\|_{C^\ell}\htop(g_i y)^\omega e^{-\delta_1t}\\
 &\ll_\varphi
 \|F\|_{C^\ell}e^{-(\delta_1-\omega\theta)t}.
\end{align*}
Since $\#\mathcal I_q\asymp J_q$, the contribution of $\mathcal G$ is $\ll_\varphi\|F\|_{C^\ell}e^{-(\delta_1-\omega\theta)t}$.

We now estimate the contribution of $\mathcal B$. Since $0\leq\psi_i\leq1$ and $\supp\psi_i\subset B_1^Hg_i$,
\[
\frac{1}{J_q}\sum_{i\in\mathcal B} \left| \int_HF(h)\psi_i(h) \bigl(\varphi(a^thy)-m_X(\varphi)\bigr)\dd m_H(h) \right|
\ll_\varphi \frac{\|F\|_{\infty} \#\mathcal B}{J_q}.
\]
Thus it remains to estimate $\#\mathcal B$.  Recall that the sets $B_{1/10}^Hg_i$ are pairwise disjoint.
For $i\in\mathcal B$, \eqref{eq:contraction} gives
\[
 B_{1/10}^Hg_i \subset B_1^H\lambda_q(B_1^H) =\lambda_q\left(\lambda_{-q}(B_1^H)B_1^H\right) \subset\lambda_q(B_2^H).
\]
By \eqref{eq:height-compact-change}, there is $c_0>0$, independent of $q,t,y$, such that $\htop(hy)\geq c_0\htop(g_i y)$ for every $h\in B_{1/10}^Hg_i$.
Therefore
\[
 \bigsqcup_{i\in\mathcal B}B_{1/10}^Hg_i
 \subset
 \left\{h\in\lambda_q(B_2^H):
 \htop(hy)>c_0e^{\theta t}\right\}.
\]
Let $ h_*=\lambda_q(u_*)$ for some $u_*\in B_1^H$.
Proposition~\ref{prop:quantitative-nondivergence} with $n=q$, $x=h_*y\in\supp\varphi$, $K=\supp\varphi$, and $R=3$ gives
\begin{align*}
 \#\mathcal B\,m_H(B_{1/10}^H) &\leq m_H\left\{h\in\lambda_q(B_2^H): \htop(hy)>c_0e^{\theta t}\right\}\\
 &=J_qm_H\left\{v\in B_2^Hu_*^{-1}: \htop(\lambda_q(v)h_*y)>c_0e^{\theta t}\right\}\\
 & \ll_\varphi J_qe^{-\kappa_1\theta t}.
\end{align*}
Therefore the contribution of $\mathcal B$ is $\ll_\varphi\|F\|_{\infty}e^{-\kappa_1\theta t}$.

Combining the two estimates, the left-hand side of
\eqref{eq:compact-return-average} is
\[
 \ll_\varphi\|F\|_{C^\ell}
 \left(e^{-\kappa_1\theta t}
 +e^{-(\delta_1-\omega\theta)t}\right).
\]
Now choose
\[
 \theta=\frac{\delta_1}{\omega+\kappa_1},
 \qquad
 \delta_2=\kappa_1\theta
 =\delta_1-\omega\theta
 =\frac{\delta_1\kappa_1}{\omega+\kappa_1}.
\]
This proves \eqref{eq:compact-return-average} for all $t\geq0$.
\end{proof}

\section{The martingale}\label{sec:martingale}

In this section we use the contracting convolutions from
Section~\ref{sec:convolution} to construct a martingale for the
randomized orbit sum.  We estimate its initial conditional expectation and
prove that its martingale differences are uniformly bounded.

Fix $N\geq1$ and $x\in X$.  We work on the product probability space
\[
 \bigl(H^N,\mathcal B(H)^{\otimes N},\nu^{\otimes N}\bigr).
\]
For $0\leq j<N$, let $h_j:H^N\to H$ be the $j$th coordinate projection.
Thus $h_0,\ldots,h_{N-1}$ are independent and have common law $\nu$.
All expectations below are taken on this probability space.  Set $P_0=e$ and,
for $1\leq k\leq N$, define $P_k:H^k\to H$ by
\begin{equation*}
 P_k(u_0,\ldots,u_{k-1})
 =\lambda_{-(k-1)}(u_{k-1})\cdots\lambda_{-1}(u_1)u_0.
\end{equation*}
Thus the arguments are in chronological order.  We use the same notation
for the random variable
\[
 P_k=P_k(h_0,\ldots,h_{k-1}).
\]
For $z\in X$, let
\[
 S_N(z)=\sum_{s=0}^{N-1}\varphi(a^s z).
\]
Let $\cF_0$ be the trivial $\sigma$-algebra and, for $1\leq k\leq N$, let
\[
 \cF_k=\sigma(h_0,\ldots,h_{k-1}).
\]
Thus $\cF_k$ records the first $k$ random variables, and $P_k$ is $\cF_k$-measurable.  Define
\[
 X_k=\E\bigl(S_N(P_Nx)\mid\cF_k\bigr),
 \qquad 0\leq k\leq N.
\]
Hence, $X_N=S_N(P_Nx)$, and for $0\leq k<N$
\[
 X_k
 =\int_{H^{N-k}}
 S_N\bigl(P_N(h_0,\ldots,h_{N-1})x\bigr)
 \dd\nu (h_{N-1})\cdots \dd\nu (h_k).
\]
The tower property gives
\[
 \E(X_{k+1}\mid\cF_k)=X_k,
 \qquad 0\leq k<N,
\]
so $(X_k,\cF_k)_{k=0}^N$ is a martingale.

\subsection{The initial conditional expectation}

\begin{prop}\label{prop:initial}
\begin{equation*}
 |X_0-Nm_X(\varphi)|
 \ll_\varphi\bigl(1+\log\htop(x)\bigr).
\end{equation*}
\end{prop}

\begin{proof}
By the definition of $X_0$, writing
$\mathbf h=(h_0,\ldots,h_{N-1})$, and recalling the definitions of $\nu_j$
and convolution from Section~\ref{sec:convolution}, we have
\begin{align*}
 X_0
 &=\int_{H^N}S_N(P_N(\mathbf h)x)
 \dd\nu^{\otimes N}(\mathbf h)\\
 &=\int_HS_N(hx)
 \dd(\nu_{N-1}*\cdots*\nu_0)(h)\\
 &=\int_HS_N(hx)\dd\mu_0^{(N)}(h).
\end{align*}
Lemma~\ref{lem:uniform-densities}, applied with $k=s=0$, and hence $q=0$ and $J_q=1$, gives a density $F_N$ such that
\[
 \frac{\dd\mu_0^{(N)}}{\dd m_H}(h)=F_N(h),
 \qquad
 \supp F_N\subset B_1^H,
 \qquad
 \|F_N\|_{C^\ell}\ll_\ell1.
\]
Consequently,
\begin{align*}
 X_0-Nm_X(\varphi)
 &=\int_H S_N(hx)F_N(h)\dd m_H(h)-Nm_X(\varphi)\\
 &=\sum_{s=0}^{N-1}\left(
 \int_H\varphi(a^shx)F_N(h)\dd m_H(h)-m_X(\varphi)
 \right),
\end{align*}
Proposition~\ref{prop:effective-equidistribution}, applied with $f=F_N$
and $n=s$, gives
\[
 \left|\int_H\varphi(a^shx)F_N(h)\dd m_H(h)-m_X(\varphi)\right|
 \ll_\varphi\htop(x)^\omega e^{-\delta_1s}.
\]
Since $F_N\dd m_H$ is a probability measure, the same expression is at most
$2\|\varphi\|_\infty$.  Combining these two bounds gives
\[
 \left|\int_H\varphi(a^shx)F_N(h)\dd m_H(h)-m_X(\varphi)\right|
 \ll_\varphi\min\left\{1,\htop(x)^\omega e^{-\delta_1 s}\right\}.
\]
Put
\[
 s_x=\left\lceil\frac\omega{\delta_1}\log\htop(x)\right\rceil.
\]
Since $\htop(x)\geq1$, we have $s_x\geq0$, and the definition of $s_x$ gives
\[
 \htop(x)^\omega e^{-\delta_1s_x}\leq1.
\]
For $0\leq s<s_x$ we use the trivial bound, while for $s\geq s_x$ we use the exponentially decaying bound.  Hence
\begin{align*}
 |X_0-Nm_X(\varphi)|
 &\ll_\varphi\left(s_x
 +\sum_{s=s_x}^{\infty}
 \htop(x)^\omega e^{-\delta_1 s}\right).
\end{align*}
Therefore
\[
 |X_0-Nm_X(\varphi)|
 \ll_\varphi\bigl(1+\log\htop(x)\bigr).
\]
\end{proof}

\subsection{Uniform martingale differences}
The increment $X_{k+1}-X_k$ measures the effect of revealing $h_k$.
Decompose
\[
 X_k=\sum_{s=0}^{N-1}X_k^s,
 \qquad
 X_k^s=\E\bigl(\varphi(a^sP_Nx)\mid\cF_k\bigr).
\]
For $s\leq k$, the dependence on $h_k$ is contracted by $\lambda_{s-k}$.  For $s>k$, we start at the first time at which the remaining variables can give a nonzero contribution and then apply Proposition~\ref{prop:compact-return-average}.  Both contributions are summable.

\begin{prop}\label{prop:bounded-differences}
For every $0\leq k<N$,
\begin{equation}\label{eq:bounded-differences}
 |X_{k+1}-X_k|\ll_\varphi 1
\end{equation}
almost surely.
\end{prop}

\begin{proof}

Fix $0\leq k<N$.  Outside a null set, fix a realization $h_0,\ldots,h_k\in\supp\nu$.  Write $\mathbf v=(v_k,\ldots,v_{N-1})$ and $\mathbf u=(u_{k+1},\ldots,u_{N-1})$, using the usual conventions for empty tuples and product measures.  For every $0\leq s<N$,
\begin{align}
 X_k^s
 &=\int_{H^{N-k}}
 \varphi\bigl(a^sP_N(h_0,\ldots,h_{k-1},\mathbf v)x\bigr)
 \dd\nu^{\otimes(N-k)}(\mathbf v),
 \label{eq:Xks-integral}\\
 X_{k+1}^s
 &=\int_{H^{N-k-1}}
 \varphi\bigl(a^sP_N(h_0,\ldots,h_k,\mathbf u)x\bigr)
 \dd\nu^{\otimes(N-k-1)}(\mathbf u).
 \label{eq:Xk1s-integral}
\end{align}
Splitting the orbit times at $s=k$, we obtain
\begin{equation}\label{eq:martingale-difference-split}
 X_{k+1}-X_k
 =
 \sum_{s=0}^{k}(X_{k+1}^s-X_k^s)
 +
 \sum_{s=k+1}^{N-1}(X_{k+1}^s-X_k^s).
\end{equation}
The first sum is controlled by contraction, while the second is controlled by Proposition~\ref{prop:compact-return-average} from the first time at which the integrand can be nonzero.

\emph{Past times.}
For $0\leq s\leq k$, couple \eqref{eq:Xks-integral} and \eqref{eq:Xk1s-integral} using the same future coordinates $\mathbf u=(u_{k+1},\ldots,u_{N-1})$, and let $v$ denote the $k$th coordinate in \eqref{eq:Xks-integral}.  It suffices to bound uniformly in $v$ and $\mathbf u$ the quantity
\begin{equation}\label{eq:function-difference}
|\varphi\left(a^s P_N(h_0,\dots,h_{k-1},h_k,\mathbf u) x\right) - \varphi\left(a^s P_N(h_0,\dots,h_{k-1},v,\mathbf u) x\right) |
\end{equation}
By right invariance and the definition of the quotient metric,
\begin{align*}
&d_X\left(a^s P_N(h_0,\dots,h_{k-1},h_k,\mathbf u) x, a^s P_N(h_0,\dots,h_{k-1},v,\mathbf u) x\right) \\
&\quad \leq d_H\left(a^s P_N(h_0,\dots,h_{k-1},h_k,\mathbf u) P_N(h_0,\dots,h_{k-1},v,\mathbf u)^{-1} a^{-s},e\right)
\end{align*}
Let
\[
 B=\lambda_{s-(N-1)}(u_{N-1})\cdots
   \lambda_{s-(k+1)}(u_{k+1}),
\]
with $B=e$ when $k=N-1$.  Cancelling the common past factors gives
\begin{align*}
&a^s P_N(h_0,\dots,h_{k-1},h_k,\mathbf u) P_N(h_0,\dots,h_{k-1},v,\mathbf u)^{-1} a^{-s} \\
 &\quad=B\lambda_{s-k}(h_kv^{-1})B^{-1}.
\end{align*}
Fix $\chi>0$ as in Lemma~\ref{lem:conjugation-adjoint}(ii).
Since $u_j,h_k,v\in\supp\nu$ and this support is compact, right
invariance and Lemma~\ref{lem:conjugation-adjoint}(ii) give
\[
 d_H(B,e)\ll\sum_{j=k+1}^{N-1}e^{-\chi(j-s)}\ll1,
 \qquad
 d_H\bigl(\lambda_{s-k}(h_kv^{-1}),e\bigr)
 \ll e^{-\chi(k-s)}.
\]
Thus $B$ lies in a fixed compact subset of $H$, and
Lemma~\ref{lem:conjugation-adjoint}(i), applied with $g=B$, yields
\[
 d_H\bigl(B\lambda_{s-k}(h_kv^{-1})B^{-1},e\bigr)
 \ll e^{-\chi(k-s)}.
\]
Using \eqref{eq:function-difference} and integrating as in
\eqref{eq:Xks-integral}--\eqref{eq:Xk1s-integral}, we obtain
\[
 |X_{k+1}^s-X_k^s|  \ll\|\varphi\|_{\Lip}e^{-\chi(k-s)}, \text{ hence}  \quad \sum_{s=0}^k|X_{k+1}^s-X_k^s|\ll_\varphi1.
\]

\emph{Future times.}
Let $s_0$ be the least $k+1\leq s < N$ for which
\begin{equation}\label{eq:first-nonzero-condition} \varphi\bigl( a^sP_N(h_0,\ldots,h_{k-1},\mathbf v)x \bigr)\neq0
 \qquad
 \text{for some }\mathbf v\in(\supp\nu)^{N-k}.
\end{equation}
If no such $s$ exists, put $s_0=N$.  It follows from \eqref{eq:Xks-integral} and \eqref{eq:Xk1s-integral} that
\[
 X_k^s=X_{k+1}^s=0,
 \qquad k<s<s_0.
\]
If $s_0=N$, the second sum in \eqref{eq:martingale-difference-split} is zero, so the past-time estimate proves the proposition.  Hence assume $s_0<N$ and choose $\mathbf v^*=(v_k^*,\ldots,v_{N-1}^*)$ satisfying \eqref{eq:first-nonzero-condition} with $s=s_0$. Recall that $P_k=P_k(h_0,\ldots,h_{k-1})$.  For $s_0\leq s<N$, we have
\begin{equation*}
 a^sP_Nx = a^{s-s_0} \lambda_{s_0}(P_NP_k^{-1}) a^{s_0}P_kx.
\end{equation*}
Hence \eqref{eq:Xks-integral} can be rewritten as
\begin{equation*}
 X_k^s = \int \varphi\Bigl( a^{s-s_0} \lambda_{s_0}\bigl( P_N(h_0,\ldots,h_{k-1},\mathbf v)P_k^{-1} \bigr) a^{s_0}P_kx \Bigr)
 \dd\nu^{\otimes(N-k)}(\mathbf v).
\end{equation*}
The same identity rewrites \eqref{eq:Xk1s-integral}, with
$(h_k,\mathbf u)$ in place of $\mathbf v$.

For the witnessing tuple $\mathbf v^*$ at time $s_0$,
\begin{align*}
 0 &\neq \varphi\bigl( a^{s_0}P_N(h_0,\ldots,h_{k-1},\mathbf v^*)x \bigr)\\
 &= \varphi\Bigl( \lambda_{s_0}\bigl( P_N(h_0,\ldots,h_{k-1},\mathbf v^*)P_k^{-1} \bigr) a^{s_0}P_kx \Bigr).
\end{align*}
Set
\[
 q=s_0-k,\qquad y=a^{s_0}P_kx,\qquad h_*= \lambda_{s_0}\left( P_N(h_0,\ldots,h_{k-1},\mathbf v^*)P_k^{-1} \right).
\]
Since
\[
 P_N(h_0,\ldots,h_{k-1},\mathbf v^*)P_k^{-1} = \lambda_{-(N-1)}(v_{N-1}^*)\cdots \lambda_{-k}(v_k^*) \in\supp\mu_k^{(N)},
\]
it follows from Lemma~\ref{lem:uniform-densities} with $s=s_0$ that
\begin{equation}\label{eq:future-return}
h_*\in\lambda_q(B_1^H) \quad\text{and}\quad h_* y \in \supp\varphi.
\end{equation}
Also, we have that for $s_0\leq s<N$,
\begin{align*}
 X_k^s
 &=\int_H\varphi\bigl(a^{s-s_0}hy\bigr)
 \dd\bigl((\lambda_{s_0})_*\mu_k^{(N)}\bigr)(h),\\
 X_{k+1}^s
 &=\int_H\varphi\bigl(a^{s-s_0}hy\bigr)
 \dd\Bigl((\lambda_{s_0})_*
 \bigl(\mu_{k+1}^{(N)}*
 \delta_{\lambda_{-k}(h_k)}\bigr)\Bigr)(h).
\end{align*}
Lemma~\ref{lem:uniform-densities},
now applied with $s=s_0$, gives $F_k,F_{k+1}\in C_c^\infty(H)$ such that
\begin{align}
 \frac{\dd(\lambda_{s_0})_*\mu_k^{(N)}}{\dd m_H}
 &=J_q^{-1}F_k,
 \label{eq:future-density-k}\\
 \frac{\dd(\lambda_{s_0})_*
 \bigl(
 \mu_{k+1}^{(N)}*
 \delta_{\lambda_{-k}(h_k)}
 \bigr)}{\dd m_H}
 &=J_q^{-1}F_{k+1},
 \label{eq:future-density-k1}
\end{align}
and, for $i\in\{k,k+1\}$,
\begin{equation*}
 F_i\geq0,\qquad
 \supp F_i\subset\lambda_q(B_1^H),\qquad
 \int_HF_i\dd m_H=J_q,\qquad
 \|F_i\|_{C^\ell}\ll_\ell1.
\end{equation*}

Substituting \eqref{eq:future-density-k}--\eqref{eq:future-density-k1}
into the preceding two integrals gives
\begin{equation}\label{eq:future-conditional-averages}
 X_i^s
 =
 \frac1{J_q}\int_H
 F_i(h)\varphi\bigl(a^{s-s_0}hy\bigr)\dd m_H(h),
 \qquad i\in\{k,k+1\}.
\end{equation}
By \eqref{eq:future-return}, the same
$h_*$ verifies the compact-return hypothesis for both integrals in
\eqref{eq:future-conditional-averages}.  Proposition~\ref{prop:compact-return-average}
with $t=s-s_0$ therefore gives
\[
 |X_i^s-m_X(\varphi)|
 \ll_\varphi
 \|F_i\|_{C^\ell}e^{-\delta_2(s-s_0)}
 \ll_\varphi e^{-\delta_2(s-s_0)},
 \qquad i\in\{k,k+1\}.
\]
Hence
\[
 |X_{k+1}^s-X_k^s|
 \ll_\varphi e^{-\delta_2(s-s_0)},
 \qquad s_0\leq s<N,
\]
and therefore
\[
 \sum_{s=k+1}^{N-1}|X_{k+1}^s-X_k^s|
 =
 \sum_{s=s_0}^{N-1}|X_{k+1}^s-X_k^s|
 \ll_\varphi
 \sum_{r=0}^{\infty}e^{-\delta_2r}
 \ll_\varphi1.
\]
Together with the past-time estimate, this proves \eqref{eq:bounded-differences}.
\end{proof}

\section{Proof of the large-deviation theorem}\label{sec:ld-proof}

In this section we apply Azuma's inequality to the martingale constructed in
Section~\ref{sec:martingale}.  We first obtain a sub-Gaussian deviation bound
on a height-truncated subset of $X$, and then use Haar invariance together
with the cusp-tail estimate to deduce the large-deviation estimate for
$m_X$.

\begin{prop}\label{prop:azuma}
Let $(X_k,\cF_k)_{k=0}^N$ be a martingale, and suppose that
$d_0,\allowbreak\ldots,\allowbreak d_{N-1}>0$ satisfy
\[
 |X_{k+1}-X_k|\leq d_k,
 \qquad 0\leq k<N,
\]
almost surely.  Then, for every $t>0$,
\[
 \Prob\{|X_N-X_0|>t\}
 \leq
 2\exp\left(-\frac{t^2}{2\sum_{k=0}^{N-1}d_k^2}\right).
\]
\end{prop}

\begin{proof}
This is the standard Azuma inequality \cite{Azuma1967weighted}.
\end{proof}

\begin{proof}[Proof of Theorem~\ref{thm:large-deviation}]
By Proposition~\ref{prop:initial}, choose $A_\varphi\geq1$, independent of
$N$ and $x$, such that
\[
 |X_0-Nm_X(\varphi)|
 \leq A_\varphi\bigl(1+\log\htop(x)\bigr).
\]
Choose $c_\varphi>0$ sufficiently small that $\frac{2A_\varphi c_\varphi}{\kappa_2}\leq\frac{1}{4}$,
fix $0<\varepsilon<1$, and set
\[
 \mathcal K_{N,\varepsilon}
 =\left\{x:\htop(x)\leq
 \exp\left(\frac{2c_\varphi\varepsilon N}{\kappa_2}\right)\right\}.
\]
For any $N\geq4A_\varphi\varepsilon^{-1}$ and
$x\in \mathcal K_{N,\varepsilon}$,
\begin{align*}
 |X_0-Nm_X(\varphi)|
 &\leq A_\varphi
 \left(1+\frac{2c_\varphi\varepsilon N}{\kappa_2}\right)\\
 &\leq A_\varphi+\frac{\varepsilon N}{4}
 \leq\frac{\varepsilon N}{2}.
\end{align*}
By Proposition~\ref{prop:bounded-differences}, choose $B_\varphi\geq1$,
independent of $N$, $x$, and $k$, such that
\[
 |X_{k+1}-X_k|\leq B_\varphi,
 \qquad 0\leq k<N,
\]
almost surely. Since $X_N=S_N(P_Nx)$, Proposition~\ref{prop:azuma} with $d_k=B_\varphi$ and
$t=\varepsilon N/2$ gives that for every $x\in\mathcal K_{N,\varepsilon}$,
\begin{equation}\label{eq:azuma-fixed-base}
\begin{split}
 \Prob\{|S_N(P_Nx)-Nm_X(\varphi)|>\varepsilon N\}
 &\leq \Prob\{|X_N-X_0|>\varepsilon N/2\}\\
 &\leq 2\exp\left(-\frac{\varepsilon^2N}{8B_\varphi^2}\right).
\end{split}
\end{equation}
Decrease $c_\varphi$, if necessary, so that $2c_\varphi\leq\frac{1}{8B_\varphi^2}$.
Then
\begin{equation}\label{eq:azuma-use}
 \Prob\{|S_N(P_Nx)-Nm_X(\varphi)|>\varepsilon N\}
 \leq2e^{-2c_\varphi\varepsilon^2N}
\end{equation}
for every $N\geq4A_\varphi\varepsilon^{-1}$ and
$x\in\mathcal K_{N,\varepsilon}$.

Let
\[
 E_{N,\varepsilon}
 =\{x:|S_N(x)-Nm_X(\varphi)|>\varepsilon N\}.
\]
For every $\mathbf h\in H^N$, left translation by
$P_N(\mathbf h)$ preserves $m_X$.  Hence Fubini's theorem gives
\begin{align*}
 \int_X\Prob\{P_Nx\in E_{N,\varepsilon}\}\dd m_X(x)
 &=
 \int_{H^N}m_X\bigl(P_N(\mathbf h)^{-1}E_{N,\varepsilon}\bigr)
 \dd\nu^{\otimes N}(\mathbf h)\\
 &=m_X(E_{N,\varepsilon}).
\end{align*}
Splitting this integral over $\mathcal K_{N,\varepsilon}$ and its
complement and using \eqref{eq:azuma-use}, we obtain
\begin{equation}\label{eq:large-deviation-split}
 m_X(E_{N,\varepsilon})
 \leq2e^{-2c_\varphi\varepsilon^2N}
 +m_X(\mathcal K_{N,\varepsilon}^{\,c}).
\end{equation}
Set
\[
 r=\exp\left(-\frac{2c_\varphi\varepsilon N}{\kappa_2}\right).
\]
Proposition~\ref{prop:cusp-tail} yields
\begin{align*}
 m_X(\mathcal K_{N,\varepsilon}^{\,c})
 &\ll r^{\kappa_2}
 =e^{-2c_\varphi\varepsilon N}
 \leq e^{-2c_\varphi\varepsilon^2N},
\end{align*}
where the last inequality uses $0<\varepsilon<1$.  Thus
\eqref{eq:large-deviation-split} gives
\[
 m_X(E_{N,\varepsilon})
 \ll_\varphi e^{-2c_\varphi\varepsilon^2N}.
\]
Let $C\geq1$ denote the fixed implied constant.  After enlarging the
threshold so that $c_\varphi\varepsilon^2N\geq\log C$, we obtain
\[
 m_X(E_{N,\varepsilon})
 \leq e^{-c_\varphi\varepsilon^2N}.\qedhere
\]
\end{proof}

\begin{proof}[Proof of Corollary~\ref{cor:continuous-ld}]
Put $\Phi(x)=\int_0^1\varphi(a_sx)\dd s$. Then
$\Phi\in C_c^\infty(X)$ and $m_X(\Phi)=m_X(\varphi)$.
For $N=\lfloor T\rfloor$, splitting into unit intervals gives
\[
 \int_0^T\bigl(\varphi(a_tx)-m_X(\varphi)\bigr)\dd t
 =\sum_{j=0}^{N-1}\bigl(\Phi(a^jx)-m_X(\Phi)\bigr)
 +O_\varphi(1),
\]
uniformly in $x$. For $T\gg_\varphi\varepsilon^{-2}$, the remainder
is at most $\varepsilon T/2$, so a continuous-time deviation greater
than $\varepsilon$ implies a discrete deviation greater than
$\varepsilon/2$. Applying Theorem~\ref{thm:large-deviation} to $\Phi$
and using $N\geq T/2$ proves the claim.
\end{proof}

Next we prove our application to pointwise equidistribution along horospheres.
\begin{proof}[Proof of Theorem~\ref{thm:pointwise-equidistribution}]
We first prove~(ii).
Take $\varepsilon_N=K\sqrt{\log N/N}$, with $K>0$ chosen so that
$c_\varphi K^2>1$. For all sufficiently large $N$, we have
$0<\varepsilon_N<1$ and $N\gg_\varphi\varepsilon_N^{-2}$, since
$K^2\log N\to\infty$. Corollary~\ref{cor:continuous-ld} therefore gives
\[
 m_X\left\{x:
 \left|\frac1N\int_0^N\varphi(a_tx)\dd t-m_X(\varphi)\right|
 >\varepsilon_N\right\}
 \leq N^{-c_\varphi K^2}.
\]
Since the right-hand side is summable, Borel--Cantelli gives the
desired estimate at integer times for $m_X$-almost every $x$.
For $N=\lfloor T\rfloor\geq1$ and every $x\in X$,
\[
 \left|\frac1T\int_0^T\varphi(a_tx)\dd t
 -\frac1N\int_0^N\varphi(a_tx)\dd t\right|
 \leq\frac{2\|\varphi\|_\infty}{T}.
\]
Since $N\asymp T$, the same estimate holds for all sufficiently large
real $T$, proving assertion~(ii).

To prove assertion~(i), put
$\Phi(x)=\int_0^1\varphi(a_sx)\dd s$. As in the proof of
Corollary~\ref{cor:continuous-ld}, we have
$\Phi\in C_c^\infty(X)$, $m_X(\Phi)=m_X(\varphi)$, and
\[
 \int_0^N\varphi(a_tx)\dd t
 =\sum_{j=0}^{N-1}\Phi(a^jx).
\]

Fix $x\in X$ and apply the martingale of
Section~\ref{sec:martingale} to $\Phi$.  By Propositions~\ref{prop:initial}
and~\ref{prop:bounded-differences}, choose $A_\Phi,B_\Phi\geq1$,
independent of $N$ and $x$, such that
\[
 |X_0-Nm_X(\Phi)|\leq A_\Phi(1+\log\htop(x)),
 \qquad |X_{k+1}-X_k|\leq B_\Phi.
\]
The condition $N\varepsilon\geq2A_\Phi(1+\log\htop(x))$
therefore ensures that $|X_0-Nm_X(\Phi)|\leq N\varepsilon/2$.
For the convolution measure $\mu_0^{(N)}$ defined in
Section~\ref{sec:convolution}, the definition of $P_N$ gives
\[
 (P_N)_*\nu^{\otimes N}
 =\nu_{N-1}*\cdots*\nu_0=\mu_0^{(N)}.
\]
Under this condition, repeating the Azuma argument for
\eqref{eq:azuma-fixed-base} with $\varphi=\Phi$ and using
this pushforward identity, we obtain

\begin{equation}\label{eq:pointwise-randomized-tail}
 \mu_0^{(N)}\left\{h:
 \left|\frac1N\sum_{j=0}^{N-1}\Phi(a^jhx)-m_X(\Phi)\right|
 >\varepsilon\right\}
 \leq2\exp\left(-\frac{\varepsilon^2N}{8B_\Phi^2}\right).
\end{equation}
Since $x$ is fixed and
$N\varepsilon_N=K\sqrt{N\log N}\to\infty$, we have
$N\varepsilon_N\geq2A_\Phi(1+\log\htop(x))$
for all sufficiently large $N$.

We need to compare $\mu_0^{(N)}$ with $m_H$ on a fixed open set.
Let $F_N=\dd\mu_0^{(N)}/\dd m_H$. By
Lemma~\ref{lem:uniform-densities}, the $F_N$ have integral $1$,
support in $B_1^H$, and uniformly bounded $C^1$ norms.
Since $\mu_0^{(N+1)}=\nu_N*\mu_0^{(N)}$, we have
\[
 F_{N+1}(h)=\int_HF_N(u^{-1}h)\dd\nu_N(u).
\]
The uniform $C^1$ bound and $d_H(u^{-1}h,h)=d_H(u,e)$ give, by
Lemma~\ref{lem:conjugation-adjoint}(ii),
\[
 \|F_{N+1}-F_N\|_\infty
 \ll\sup_{u\in\supp\nu_N}d_H(u,e)
 \ll e^{-\chi N}.
\]
Summing this bound, choose $N_0$ large enough that
$\|F_N-F_{N_0}\|_\infty<1/4$ for all $N\geq N_0$.
Since $m_H(B_1^H)=1$, we have $\max F_{N_0}\geq1$.
By continuity, choose a nonempty relatively compact open set $V$
on which $F_{N_0}>3/4$. Then, for every $h\in V$ and $N\geq N_0$,
\[
 F_N(h)\geq\tfrac12.
\]

Combining this lower bound with \eqref{eq:pointwise-randomized-tail},
we obtain, for fixed $x$,
\[
 m_H\left\{h\in V:
 \left|\frac1N\int_0^N\varphi(a_thx)\dd t-m_X(\varphi)\right|
 >\varepsilon_N\right\}
 \leq 4N^{-K^2/(8B_\Phi^2)}
\]
for all sufficiently large $N$.
Increase $K$, if necessary, so that $K^2/(8B_\Phi^2)>1$.
For each fixed $x$, Borel--Cantelli then gives the desired estimate at
integer times for $m_H$-almost every $h\in V$.
Apply this estimate with base point $h_0x$ for each member $Vh_0$
of a countable cover of $H$. Right invariance of $m_H$ then gives
the estimate for $m_H$-almost every $h\in H$.
The comparison above extends it to all sufficiently large real $T$.
\end{proof}

\section{Entropy and effective uniqueness}\label{sec:entropy}

We derive Theorem~\ref{thm:entropy-main} by combining two estimates for
two-sided Bowen balls with Theorem~\ref{thm:large-deviation}.

For $N\geq0$ and $\rho>0$, set
\[
 B_{N,\rho}^G
 =\bigcap_{n=-N}^{N}a^{-n}B_\rho^G a^n.
\]
A two-sided Bowen ball in $X$ is a set of the form $B_{N,\rho}^Gx$, where
$x\in X$.

The Brin--Katok local entropy theorem relates entropy to the asymptotic
mass of Bowen balls on compact metric spaces \cite{BK83local}.
Since $X$ need not be compact, we use Khayutin's two-sided local-entropy
lower bound for second countable metric spaces
\cite[Theorem~A.1]{Khayutin2017Large}.  The argument of
\cite[Lemma~6.4]{Khayutin2017Large} makes this bound uniform on a compact
set of arbitrarily large measure, as follows.

\begin{prop}\label{prop:local-entropy}
Let $\mu$ be an ergodic $a$-invariant probability measure.  For every
$\eta>0$ there are a compact $Y_\eta\subset X$, a radius
$\rho_\eta>0$, and $N_\eta$ such that
$\mu(Y_\eta)>1-\eta$ and, for every $y\in Y_\eta$ and $N\geq N_\eta$,
\begin{equation*}
 \mu(B_{N,\rho_\eta}^Gy)
 \leq e^{-2N(h_\mu(a)-\eta)}.
\end{equation*}
\end{prop}

\begin{proof}
This follows from \cite[Lemma~6.4 and Theorem~A.1]{Khayutin2017Large}.
\end{proof}

We also require the following estimate for the Haar measure of Bowen balls.
\begin{prop}\label{prop:bowen-volume}
For every compact $Y\subset X$ there is $\rho_Y>0$ such that,
for $0<\rho<\rho_Y$, $N\geq0$, and $y\in Y$,
\begin{equation}\label{eq:bowen-volume}
 m_X(B_{N,\rho}^Gy)
 \asymp_Y\rho^{\dim G}e^{-2Nh_{m_X}(a)}.
\end{equation}
\end{prop}

\begin{proof}
The argument is analogous to \cite[Proposition~B.1]{Khayutin2017Large}; see
also \cite[Lemma~7.2]{EinsiedlerKadyrovPohl2015} in real rank one.

Recall from Subsection~\ref{subsec:horospherical} that
$\mathfrak g=\bigoplus_\beta E_\beta$, with the eigenspaces mutually
orthogonal.
By compactness of $Y$, choose $\rho_Y>0$ so that $\pi_y$ is injective on
$B_{\rho_Y}^G$ for every $y\in Y$.  After decreasing $\rho_Y$, we claim
that there are constants $0<c<C$ such that
\begin{multline*}
 \exp\left\{\sum_\beta v_\beta:
 \|v_\beta\|\leq c\rho e^{-N|\log|\beta||}\right\}
 \subset B_{N,\rho}^G\\
 {}\subset \exp\left\{\sum_\beta v_\beta:
 \|v_\beta\|\leq C\rho e^{-N|\log|\beta||}\right\}.
\end{multline*}

Indeed, since the exponential map is a local diffeomorphism, there are
$0<c_0<C_0$ such that, for all sufficiently small $\rho$,
\[
 \exp(B_0^{\mathfrak g}(c_0\rho))\subset B_\rho^G
 \subset\exp(B_0^{\mathfrak g}(C_0\rho)).
\]
Choose a neighborhood $U$ of $0$ on which $\exp$ is injective, and shrink
$\rho_Y$ so that $B_0^{\mathfrak g}(C_0\rho_Y)$ and its
$\Ad(a^{\pm1})$-images lie in $U$.
Also, $a^n\exp(v)a^{-n}=\exp(\Ad(a^n)v)$.
Writing $v=\sum_\beta v_\beta$, we have
\[
 a^n\exp(v)a^{-n}=\exp\left(\sum_\beta\beta^nv_\beta\right)
\]
and
\[
 \|\Ad(a^n)v\|^2
 =\left\|\sum_\beta\beta^nv_\beta\right\|^2
 =\sum_\beta|\beta|^{2n}\|v_\beta\|^2.
\]
Now let $g\in B_{N,\rho}^G$.  For each $|n|\leq N$, there is a unique
$w_n\in B_0^{\mathfrak g}(C_0\rho)$ such that
\[
 a^nga^{-n}=\exp(w_n).
\]
Put $v=w_0$.
For consecutive $n,n+1\in[-N,N]$,
$\exp(\Ad(a)w_n)=\exp(w_{n+1})$; since
$\Ad(a)w_n,w_{n+1}\in U$ and $\exp|_U$ is injective,
$w_{n+1}=\Ad(a)w_n$.  The same argument with $a^{-1}$ gives
$w_{n-1}=\Ad(a^{-1})w_n$.  Thus $w_n=\Ad(a^n)v$ for all $|n|\leq N$, and
$\sum_\beta|\beta|^{2n}\|v_\beta\|^2<C_0^2\rho^2$ for $n=\pm N$.
Taking $n=N$ for $|\beta|>1$ and $n=-N$ otherwise gives
$\|v_\beta\|\leq C_0\rho e^{-N|\log|\beta||}$, proving the second
inclusion.

Conversely, choose $c=c_0/(2\sqrt{\dim G})$.
If $\|v_\beta\|\leq c\rho e^{-N|\log|\beta||}$ for every $\beta$, then
$\|\Ad(a^n)v\|\leq\sqrt{\dim G}\,c\rho<c_0\rho$ for all $|n|\leq N$,
since $|\beta|^n e^{-N|\log|\beta||}\leq1$ in this range.
Thus
\[
 a^n\exp(v)a^{-n}\in\exp(B_0^{\mathfrak g}(c_0\rho))
 \subset B_\rho^G,\qquad |n|\leq N,
\]
so $\exp(v)\in B_{N,\rho}^G$, proving the first inclusion.

Next, we estimate the Haar measure of these boxes.  Normalize $m_G$ to
induce $m_X$.  On a fixed small exponential chart, Haar measure has a
smooth positive density bounded above and below by positive constants, so
$m_G(\exp E)\asymp\operatorname{vol}_{\mathfrak g}(E)$; see
\cite[proof of Lemma~7.2]{EinsiedlerKadyrovPohl2015}.
After shrinking $\rho_Y$, both eigenspace boxes lie in this chart.

The local Haar density and the injectivity of $\pi_y$ on $B_{\rho_Y}^G$
therefore give
\[
 m_X(B_{N,\rho}^Gy)\asymp_Y
 \rho^{\dim G}\exp\left(-N\sum_\beta
 \dim(E_\beta)|\log|\beta||\right).
\]
Finally, unimodularity gives
\[
 0=\log|\det\Ad(a)|
 =\sum_\beta\dim(E_\beta)\log|\beta|.
\]
Together with the Haar entropy formula \cite[Section~7.8]{EL10}, this yields
\[
 \sum_\beta\dim(E_\beta)|\log|\beta||
 =2\sum_{|\beta|>1}\dim(E_\beta)\log|\beta|
 =2h_{m_X}(a),
\]
which proves \eqref{eq:bowen-volume}.
\end{proof}

\begin{proof}[Proof of Theorem~\ref{thm:entropy-main}]
First assume that $\mu$ is ergodic.
The leafwise entropy bound and the Haar entropy formula give
$h_\mu(a)\leq h_{m_X}(a)$; see
\cite[Theorem~7.6(ii), Theorem~7.9, and Section~7.8]{EL10}.  Put
\[
 \Delta=|m_X(\varphi)-\mu(\varphi)|,
 \qquad
 \tau_\varphi=\frac{1}{4\max\{1,\|\varphi\|_\infty\}}.
\]
If $\Delta=0$, there is nothing to prove.

Shifting the orbit segment by $a^N$ gives
\[
 \frac1{2N+1}\sum_{n=-N}^{N}\varphi(a^nx)
 =\frac1{2N+1}\sum_{j=0}^{2N}\varphi(a^jy),
\]
where $x=a^Ny$.  Hence Haar invariance and
Theorem~\ref{thm:large-deviation} give
\begin{equation}\label{eq:symmetric-ld}
 m_X\left\{x:\left|
 \frac1{2N+1}\sum_{n=-N}^{N}\varphi(a^nx)-m_X(\varphi)
 \right|>r\right\}
 \mathrel{\leq} e^{-c_\varphi r^2N}
\end{equation}
for every $0<r<1$ and $2N+1\gg_\varphi r^{-2}$.

Fix a small $\eta>0$.  Proposition~\ref{prop:local-entropy} gives a compact
set $Y_\eta$ and a radius $\rho_\eta>0$.  By the ergodic theorem, for all
sufficiently large $N$, the set
\[
 Z_N=\left\{x\in Y_\eta:\left|
 \frac1{2N+1}\sum_{n=-N}^{N}\varphi(a^nx)-\mu(\varphi)
 \right|<(1-\tau_\varphi)\Delta\right\}
\]
satisfies $\mu(Z_N)>1-2\eta$.

For all sufficiently small $\rho$, smoothness of inversion and multiplication
near the identity gives
$(B_\rho^G)^{-1}B_\rho^G\subset B_{3\rho}^G$.  Applying this after
conjugation for each $|n|\leq N$ gives
\[
 (B_{N,\rho}^G)^{-1}B_{N,\rho}^G
 \subset B_{N,3\rho}^G.
\]
Choose $\rho>0$, independently of $N$, so small that $3\rho<\rho_\eta$, the
estimate in Proposition~\ref{prop:bowen-volume} holds uniformly on $Y_\eta$,
and $\|\varphi\|_{\Lip}\rho<\tau_\varphi\Delta/2$.

For all sufficiently large $N$, choose a maximal subset
$\Xi_N\subset Z_N$ such that
$B_{N,\rho}^G\xi\cap B_{N,\rho}^G\xi'=\varnothing$ whenever
$\xi\neq\xi'$ in $\Xi_N$.  The enlarged balls $B_{N,3\rho}^G\xi$ cover
$Z_N$.  Indeed, maximality gives, for every $z\in Z_N$, some
$\xi\in\Xi_N$ such that
$B_{N,\rho}^Gz\cap B_{N,\rho}^G\xi\neq\varnothing$.  Writing a point of
intersection as $gz=h\xi$, we obtain
\[
 z=g^{-1}h\xi
 \in(B_{N,\rho}^G)^{-1}B_{N,\rho}^G\xi
 \subset B_{N,3\rho}^G\xi.
\]
Since $3\rho<\rho_\eta$, Proposition~\ref{prop:local-entropy} now gives
\[
 1-2\eta
 <\mu(Z_N)
 \leq\sum_{\xi\in\Xi_N}\mu(B_{N,3\rho}^G\xi)
 \leq |\Xi_N|e^{-2N(h_\mu(a)-\eta)}.
\]
It follows that
\[
 |\Xi_N|\geq(1-2\eta)e^{2N(h_\mu(a)-\eta)}.
\]
Using this estimate, the disjointness of the original balls, and
Proposition~\ref{prop:bowen-volume}, we obtain
\begin{equation}\label{eq:haar-lower}
 \begin{aligned}
 m_X\left(\bigsqcup_{\xi\in\Xi_N}B_{N,\rho}^G\xi\right)
 &\gg_{Y_\eta,\rho}|\Xi_N|e^{-2Nh_{m_X}(a)}\\
 &\gg_{Y_\eta,\rho,\eta}
 e^{-2N(h_{m_X}(a)-h_\mu(a)+\eta)}.
 \end{aligned}
\end{equation}

If $z\in B_{N,\rho}^G\xi$ with $\xi\in\Xi_N$, then
\[
 \left|\frac1{2N+1}\sum_{n=-N}^{N}
 \bigl(\varphi(a^nz)-\varphi(a^n\xi)\bigr)\right|
 \leq\|\varphi\|_{\Lip}\rho.
\]
Since $\xi\in Z_N$, it follows that
\begin{align*}
 \left|\frac1{2N+1}\sum_{n=-N}^{N}\varphi(a^nz)-m_X(\varphi)\right|
 &\geq \Delta
 -\left|\frac1{2N+1}\sum_{n=-N}^{N}\varphi(a^n\xi)-\mu(\varphi)\right|\\
 &\qquad
 -\left|\frac1{2N+1}\sum_{n=-N}^{N}
 \bigl(\varphi(a^nz)-\varphi(a^n\xi)\bigr)\right|\\
 &>\Delta-(1-\tau_\varphi)\Delta-\|\varphi\|_{\Lip}\rho\\
 &>\frac{\tau_\varphi\Delta}{2}.
\end{align*}
Thus the union in \eqref{eq:haar-lower} lies in the exceptional set in
\eqref{eq:symmetric-ld} with $r=\tau_\varphi\Delta/2$.  The choice of
$\tau_\varphi$ ensures that $0<r<1$.  For fixed $\eta$ and $\rho$, comparing
the two estimates, taking logarithms, dividing by $N$, and letting
$N\to\infty$ gives
\[
 c_\varphi\left(\frac{\tau_\varphi\Delta}{2}\right)^2
 \leq2\bigl(h_{m_X}(a)-h_\mu(a)+\eta\bigr).
\]
Letting $\eta\to0$, we obtain
\[
 \Delta^2\ll_\varphi h_{m_X}(a)-h_\mu(a).
\]

For an arbitrary invariant probability measure, write its ergodic
decomposition as $\mu=\int\mu_\xi\dd\tau(\xi)$.  Jensen's inequality, the
ergodic case, and affinity of entropy give
\begin{align*}
 |m_X(\varphi)-\mu(\varphi)|^2
 &\leq\int|m_X(\varphi)-\mu_\xi(\varphi)|^2\dd\tau(\xi)\\
 &\ll_\varphi
 \int\bigl(h_{m_X}(a)-h_{\mu_\xi}(a)\bigr)\dd\tau(\xi)\\
 &=h_{m_X}(a)-h_\mu(a).
\end{align*}
This proves \eqref{eq:entropy-main}.
\end{proof}

\raggedbottom
\providecommand{\bysame}{\leavevmode\hbox to3em{\hrulefill}\thinspace}


\begin{thebibliography}{LMMS24}

\bibitem[Azu67]{Azuma1967weighted}
Kazuoki Azuma, \emph{Weighted sums of certain dependent random variables},
  Tohoku Math. J. (2) \textbf{19} (1967), no.~3, 357--367.

\bibitem[BH26]{BH26}
Timoth{\'e}e B{\'e}nard and Weikun He, \emph{Multislicing and effective
  equidistribution for random walks on some homogeneous spaces}, To appear in
  Ann. of Math. (2), 2026.

\bibitem[Bir31]{Birkhoff1931ergodic}
George~D. Birkhoff, \emph{Proof of the ergodic theorem}, Proc. Natl. Acad. Sci.
  USA \textbf{17} (1931), no.~12, 656--660.

\bibitem[BK83]{BK83local}
M.~Brin and A.~Katok, \emph{On local entropy}, Geometric Dynamics (Rio de
  Janeiro, 1981), Lecture Notes in Math., vol. 1007, Springer, Berlin, 1983,
  pp.~30--38.

\bibitem[Bow75]{Bowen1975Equilibrium}
Rufus Bowen, \emph{Equilibrium states and the ergodic theory of {A}nosov
  diffeomorphisms}, Lecture Notes in Math., vol. 470, Springer-Verlag, Berlin,
  1975.

\bibitem[Cra38]{Cramer1938}
Harald Cram{\'e}r, \emph{Sur un nouveau th{\'e}or{\`e}me-limite de la
  th{\'e}orie des probabilit{\'e}s}, Les sommes et les fonctions de variables
  al{\'e}atoires, Actualit{\'e}s scientifiques et industrielles, vol. 736,
  Hermann, Paris, 1938, Colloque consacr{\'e} {\`a} la th{\'e}orie des
  probabilit{\'e}s, vol. 3, pp.~5--23.

\bibitem[DZ10]{DemboZeitouni2010}
Amir Dembo and Ofer Zeitouni, \emph{Large deviations techniques and
  applications}, 2nd ed., Stochastic Modelling and Applied Probability, vol.~38,
  Springer, Berlin, 2010, Corrected reprint of the second (1998) edition.

\bibitem[EKP15]{EinsiedlerKadyrovPohl2015}
Manfred Einsiedler, Shirali Kadyrov, and Anke~D. Pohl, \emph{Escape of mass and
  entropy for diagonal flows in real rank one situations}, Israel J. Math.
  \textbf{210} (2015), no.~1, 245--295.

\bibitem[EL10]{EL10}
Manfred Einsiedler and Elon Lindenstrauss, \emph{Diagonal actions on locally
  homogeneous spaces}, Homogeneous Flows, Moduli Spaces and Arithmetic, Clay
  Math. Proc., vol.~10, American Mathematical Society, Providence, RI, 2010,
  pp.~155--241.

\bibitem[Kad15]{Kadyrov2015Effective}
Shirali Kadyrov, \emph{Effective uniqueness of {P}arry measure and exceptional
  sets in ergodic theory}, Monatsh. Math. \textbf{178} (2015), no.~2, 237--249.

\bibitem[Kha17]{Khayutin2017Large}
Ilya Khayutin, \emph{Large deviations and effective equidistribution}, Int.
  Math. Res. Not. IMRN (2017), no.~10, 3050--3106.

\bibitem[Kif90]{Kifer1990large}
Yuri Kifer, \emph{Large deviations in dynamical systems and stochastic
  processes}, Trans. Amer. Math. Soc. \textbf{321} (1990), no.~2, 505--524.

\bibitem[KM96]{KM96}
Dmitry~Y. Kleinbock and Gregory~A. Margulis, \emph{Bounded orbits of
  nonquasiunipotent flows on homogeneous spaces}, Sinai's Moscow Seminar on
  Dynamical Systems, Amer. Math. Soc. Transl. Ser. 2, vol. 171, American
  Mathematical Society, Providence, RI, 1996, pp.~141--172.

\bibitem[KM99]{KM1999logarithm}
\bysame, \emph{Logarithm laws for flows on homogeneous spaces}, Invent. Math.
  \textbf{138} (1999), no.~3, 451--494.

\bibitem[KM12]{KM2012effective}
\bysame, \emph{On effective equidistribution of expanding translates
  of certain orbits in the space of lattices}, Number Theory, Analysis
  and Geometry, Springer, New York, 2012, pp.~385--396.

\bibitem[KM20]{KM20}
Dmitry Kleinbock and Shahriar Mirzadeh, \emph{Dimension estimates for the set
  of points with non-dense orbit in homogeneous spaces}, Math. Z. \textbf{295}
  (2020), no.~3--4, 1355--1383.

\bibitem[KS09]{KelmerSarnak2009}
Dubi Kelmer and Peter Sarnak, \emph{Strong spectral gaps for compact quotients
  of products of {$\operatorname{PSL}(2,\mathbb R)$}}, J. Eur. Math. Soc.
  (JEMS) \textbf{11} (2009), no.~2, 283--313.

\bibitem[KSW17]{KSW17}
Dmitry Kleinbock, Ronggang Shi, and Barak Weiss, \emph{Pointwise
  equidistribution with an error rate and with respect to unbounded functions},
  Math. Ann. \textbf{367} (2017), 857--879.

\bibitem[LMMS24]{LMMS2024quantitative}
Elon Lindenstrauss, Gregorii Margulis, Amir Mohammadi, and Nimish~A. Shah,
  \emph{Quantitative behavior of unipotent flows and an effective avoidance
  principle}, J. Anal. Math. \textbf{153} (2024), no.~1, 1--61.

\bibitem[MN08]{MelbourneNicol2008}
Ian Melbourne and Matthew Nicol, \emph{Large deviations for nonuniformly
  hyperbolic systems}, Trans. Amer. Math. Soc. \textbf{360} (2008), no.~12,
  6661--6676.

\bibitem[Par64]{Parry1964Intrinsic}
William Parry, \emph{Intrinsic {M}arkov chains}, Trans. Amer. Math. Soc.
  \textbf{112} (1964), 55--66.

\bibitem[Pol11]{Polo2011Equidistribution}
Fabrizio Polo, \emph{Equidistribution in chaotic dynamical systems}, ProQuest
  LLC, Ann Arbor, MI, 2011, Thesis (Ph.D.)--The Ohio State University.

\bibitem[RS22]{RuehrSarig2022}
Ren{\'e} R{\"u}hr and Omri Sarig, \emph{Effective intrinsic ergodicity for
  countable state {M}arkov shifts}, Israel J. Math. \textbf{251} (2022), no.~2,
  679--735.

\bibitem[R{\"u}h16]{Ruehr2016Effectivity}
Ren{\'e} R{\"u}hr, \emph{Effectivity of uniqueness of the maximal entropy
  measure on {$p$}-adic homogeneous spaces}, Ergodic Theory Dynam. Systems
  \textbf{36} (2016), no.~6, 1972--1988.

\bibitem[Sha16]{Sharp2016Survey}
Richard Sharp, \emph{Large deviations for dynamical systems}, 2016, Extended
  lecture notes for the workshop ``Statistical Properties of Nonequilibrium
  Dynamical Systems,'' SUSTech, Shenzhen, July 2016.

\bibitem[Shi21]{Shi21}
Ronggang Shi, \emph{Expanding cone and applications to homogeneous dynamics},
  Int. Math. Res. Not. IMRN (2021), no.~9, 7060--7095.

\bibitem[SW25]{SolanWieser2025}
Omri~Nisan Solan and Andreas Wieser, \emph{{Birkhoff} generic points on curves
  in horospheres}, Selecta Math. (N.S.) \textbf{31} (2025), no.~5, Paper No.
  105.

\bibitem[You90]{Young1990large}
Lai-Sang Young, \emph{Large deviations in dynamical systems}, Trans. Amer.
  Math. Soc. \textbf{318} (1990), no.~2, 525--543.

\end{thebibliography}
\end{document}